\documentclass[11pt]{article}
\usepackage[utf8]{inputenc}
\usepackage{lmodern}
\usepackage{graphicx,mathrsfs, ulem}
\usepackage[monochrome]{color}
\usepackage{amsfonts}
\usepackage{stmaryrd}
\usepackage[nice]{nicefrac}
\usepackage{amsmath,amsthm,amssymb}
\usepackage{extarrows}
\usepackage[mathlines]{lineno}
\usepackage{enumerate}
\usepackage{hyperref}
\usepackage[top=1in, bottom=1in, left=1in , right=1in]{geometry}
\usepackage{caption}
\usepackage{xifthen} 
\usepackage{wrapfig}
\usepackage[numbers, square]{natbib}
\usepackage{comment}
\usepackage{float}
\usepackage{ifthen}
\usepackage{tikz}
\usepackage{caption}
\usepackage{authblk}
\usepackage{mathtools}
\mathtoolsset{showonlyrefs}
\usetikzlibrary{arrows}
\usetikzlibrary{graphs,graphs.standard}
\usetikzlibrary{positioning,arrows.meta}
\usetikzlibrary{shapes.multipart}
\usetikzlibrary{arrows}
\usetikzlibrary{automata}

\newtheorem{theorem}{Theorem}[section]
\newtheorem{lemma}[theorem]{Lemma}
\newtheorem*{claim}{Claim}
\newtheorem{prop}[theorem]{Proposition}
\newtheorem{cor}[theorem]{Corollary}
\newtheorem{conj}{Conjecture}[section]

\theoremstyle{definition}
\newtheorem{rmq}{Remark}[section]

\newcommand{\E}[2][]{\ensuremath{\mathbb{E}_{#1} \left[#2 \right]}}

\newcommand{\Prob}[2][]{\ensuremath{\mathbb{P}_{#1} \left(#2 \right)}}
\newcommand{\red}[1]{{\color{red} #1}}
\newcommand{\Supp}[1]{\ensuremath{\text{Supp} \left( #1 \right)}}
\newcommand{\dd}{\mathrm d}
\newcommand{\N}{\mathbb{N}}

\newcommand{\eps}{\varepsilon}
\newcommand{\wmax}{w^{*}}
\newcommand{\parti}{\mathcal{Z}}

\newcommand{\Ind}{\mathcal{I}}
\newcommand{\Jind}{\mathcal{J}}
\newcommand{\cE}{\mathcal{E}}
\newcommand{\cD}{\mathcal{D}}
\newcommand{\cT}{\mathcal{T}}
\newcommand{\cZ}{\mathcal{Z}}
\newcommand{\outdeg}[1]{\ensuremath{\text{deg}^{+}(#1)}}

\DeclareMathOperator{\prog}{prog}
\DeclareMathOperator{\esssup}{ess \, sup}
\DeclareMathOperator{\Exp}{Exp}
\newcommand{\exprv}[1]{\Exp\left( #1 \right)}

\newcommand{\rif}[1]{\ensuremath{\left(#1\right)\text{-RIF tree}}}
\newcommand{\rifs}[1]{\ensuremath{\left(#1\right)\text{-RIF trees}}}
\DeclareMathOperator{\gpaf}{GPAF}

\newcommand{\redd}[1]{{\color{violet} #1}}
\newcommand{\greenn}[1]{{\color{teal} #1}}
\newcommand{\bluee}[1]{{\color{blue} #1}}

\title{Degree Distributions in Recursive Trees with Fitnesses}
\author{Tejas Iyer}
\affil{School of Mathematics, \\
University of Birmingham.\\
T.Iyer@bham.ac.uk}
\date{\today}

\begin{document}

\maketitle
\begin{abstract} 
 We study a general model of recursive trees where vertices are equipped with independent weights and at each time-step a vertex is sampled with probability proportional to its fitness function, which is a function of its weight and degree, and connects to $\ell$ new-coming vertices. Under a certain technical assumption, applying the theory of Crump-Mode-Jagers branching processes, we derive formulas for the limiting distributions of the proportion of vertices with a given degree and weight, and proportion of edges with endpoint having a certain weight. As an application of this theorem, we rigorously prove observations of Bianconi related to the evolving Cayley tree in [$\mathit{Phys. \, Rev. \, E} \; \mathbf{66}, \text{ 036116 (2002)}$]. We also study the process in depth when the technical condition~can fail in the particular case when the fitness function is affine, a model we call ``generalised preferential attachment with fitness''. We show that this model can exhibit condensation where a positive proportion of edges accumulate around vertices with maximal weight, or, more drastically, have a degenerate limiting degree distribution where the entire proportion of edges accumulate around these vertices. Finally, we  prove stochastic convergence for the degree distribution under a different assumption of a strong law of large numbers for the partition function associated with the process. 
\end{abstract}
\section{Introduction}
\normalem
\emph{Recursive trees} are \emph{rooted} labelled trees that are \emph{increasing}, that is, starting at a distinguished \emph{root} vertex labelled $0$, nodes are labelled in increasing order away from the root. Recursive trees generated using stochastic processes have attracted widespread study, motivated by, for example,  their applications to the evolution of languages \cite{najock-heyde}, the analysis of algorithms \cite{mahmoud1992evolution} and the study of complex networks (see, for example, \cite[Chapter~8.1]{RemcoRGCN}). Other applications include modelling the spread of epidemics, pyramid schemes and constructing family trees of ancient manuscripts (e.g. \cite[page~14]{drmota2009random}).

A common framework for randomly generating recursive trees is to have vertices arrive one at a time and connect to an existing vertex in the tree selected according to some probability distribution. In the \emph{uniform recursive tree}, introduced by Na and Rapoport in \cite{narapoport}, existing vertices are chosen uniformly at random, whilst the well-known random \emph{ordered recursive tree}, introduced by Prodinger and Urbanek in \cite{prodingerurbanek} may be interpreted as having existing vertices chosen with probability proportional to their degree. The latter model has been studied and rediscovered under various guises: under the name \emph{nonuniform recursive trees} by Szyma\'{n}ski in \cite{szymanski}, random \emph{plane oriented recursive trees} in \cite{mahmoud92,mahmoudetal93}, random \emph{heap ordered recursive trees} \cite{chen1994} and \emph{scale-free trees} \cite{bollobaspreferential,szaboalava2002,bollobasriordan04}.  
Random ordered recursive trees, or plane-oriented recursive trees, are so named because the process stopped after $n$ vertices arrive is distributed like a tree chosen at random from the set of rooted labelled trees on $n$ vertices embedded in the plane where descendants of a node are ordered from left to right. However, as first observed by Albert and Barab\'{a}si in \cite{Barabasi509} and studied in a mathematically precise way in \cite{bollobaspreferential,mori-trees-2002}, these trees, and more generally graphs evolving according to a similar mechanism, possess many interesting, non-trivial properties of real world networks. These properties include having a power law degree distribution with exponent between $2$ and $3$ and a diameter that scales logarithmically in the number of vertices. The latter may be interpreted as a `small-world' phenomenon: despite the size of the network being large, the diameter of the network remains relatively small. In this context, the fact that vertices are chosen according to their degree may be interpreted as the network showing `preference' for vertices of high degree, hence the model is often called \emph{preferential attachment}. This model has been generalised in a number of ways, to encompass the cases where vertices are chosen according to a \emph{super-linear} function of their degree \cite{Oliveira-spencer} and a \emph{sub-linear} function of their degree \cite{dereich-morters-sublinear}, or indeed any positive function of the degree \cite{rudas}. In \cite{holmgren-janson}, this model is generalised to possibly non-negative functions of the degree and is referred to as \emph{generalised preferential attachment}. 

In applications, it is often interesting to add weights to vertices as a measure of the intrinsic `fitness' of the node the vertex represents. In the \emph{Bianconi-Barab\'{a}si model}, or \emph{preferential attachment with multiplicative fitness}, introduced in \cite{bianconibarabasi2001} and studied in a mathematically precise way in \cite{borgs-chayes,dereich-unfolding,dereich-ortgiese, bhamidi,dereich-mailler-morters}, vertices are equipped with independent, identically distributed (i.i.d.) weights and connect to previous vertices with probability proportional to the product of their weight and their degree. Interestingly, as observed in \cite{bianconibarabasi2001} and confirmed rigorously in \cite{borgs-chayes,dereich-ortgiese,dereich-mailler-morters}, there is a critical condition~on the weight distribution under which this model undergoes a phase transition, resulting in a \emph{Bose-Einstein condensation}: in the limit a positive fraction of vertices accumulate around vertices of maximum weight. In a similar model known as \emph{preferential attachment with additive fitness} introduced in \cite{ergun} and studied mathematically in \cite{bhamidi,wrt2delphin,bas}, 
vertices connect to previous vertices with probability proportional to the sum of their degree (or degree minus one) and their weight. A number of other interesting preferential attachment models with fitness have been studied, including a related continuous time model which incorporates \emph{aging} of vertices \cite{vdh-aging-mult-fitness-2017} and discrete-time models with co-existing additive and multiplicative attachment rules \cite{antunovic_mossel_racz_2016, jonathan-1}. 

Adding weights also allows for a generalisation of the uniform recursive tree called the \emph{weighted recursive tree}, where now vertices connect to previous vertices with probability proportional to their weight. This model was introduced in \cite{wrt1}, for specific types of weights and in full generality in \cite{ella-umit2017}. It was also introduced independently by Janson in the case that all weights are one except the root, motivated by applications to infinite colour P\'{o}lya urns \cite{janson2018convergence}. In \cite{wrt2delphin}, S\'{e}nizergues showed that a preferential attachment tree with additive fitness with deterministic weights is equal in distribution to an associated weighted random recursive tree with random weights, an interesting link between the two classes of models. 

Motivated by applications to invasive percolation models in physics, Bianconi \cite{bianconi_cayley} introduced a similar model of growing \emph{Cayley trees}. In this model, vertices are equipped with independent weights and are either \emph{active} or \emph{inactive}. At each time-step an active vertex is chosen with probability proportional to its weight, produces $\ell$ new vertices with weights of their own and then becomes inactive. Bianconi observed that in this model, the distribution of weights of active vertices converges to a \emph{Fermi-Dirac distribution}, in contrast to the \emph{Bose distribution} that emerges in the Bianconi-Barab\'{a}si model. 

\subsection{Notation}
Generally in this paper we set $\mathbb{N}_{0} := \mathbb{N} \cup \{0\}$ and $\mathbb{R}_{+} := [0, \infty)$. We assume that $\mathbb{R}_{+}$ is equipped with the usual \emph{Borel sigma algebra}, and $\mu$ will denote a fixed probability measure on $\mathbb{R}_{+}$. Also, in general in this paper, $W$ refers to a generic $\mu$ distributed random variable. Finally, \redd{we generally refer to \emph{measurable set}, when referring to an element of the Borel sigma field.} Given \redd{such} a set $A$, we denote by $\mathbf{1}_{A}(x)$ the indicator function associated with this set, so that $\mathbf{1}_{A}(x) = 1$ if $x \in A$ and $0$ otherwise. Moreover, if $\mathbf{1}_{A}(x)$ is a random variable on a probability space $(\Omega, \mathbb{F}, \mathbb{P})$, we \redd{often} omit the dependence on $x \in \Omega$, and simply write $\mathbf{1}_{A}$. 

\subsection{Description of Model}
The goal of this paper is to present a unified model that encompasses most of the models described in the introduction above. 
In order to define the model, we first require a probability measure $\mu$ supported on $\mathbb{R}_{+}$ and a \emph{fitness function}, which is a function $f: \N_0 \times \mathbb{R}_{+} \rightarrow \mathbb{R}_+$. We consider evolving sequences of \textit{weighted oriented trees} $\cT := \left(\mathcal{T}_{t}\right)_{t \in \N_0}$; these are trees with \emph{directed edges}, where vertices have real valued weights assigned to them. The model also has an additional parameter $\ell \in \N$. We start with an initial tree $\mathcal{T}_0$ consisting of a single vertex $0$ with weight $W_0$ sampled from $\mu$. To ensure that the evolution of the model is well-defined, \redd{for brevity}, we assume $f(0,W_0) > 0$ almost surely. 
Then, we define $\mathcal{T}_{t+1}$ recursively as follows: 
\begin{enumerate}[(i)]
\item Sample a vertex $j$ from $\mathcal{T}_t$ with probability \[\frac{f(\outdeg{j,\mathcal{T}_t}/\ell, W_j)}{\parti_{t}
},\]
where $\outdeg{j,\mathcal{T}_t}$ denotes the out-degree of the vertex $j$ in the oriented tree $\cT_t$ and $\mathcal{Z}_t : = \sum_{j=0}^{\ell t} f(\outdeg{j,\mathcal{T}_t)/\ell, W_j}$ is the \emph{partition function} associated with the process. 
\item Introduce $\ell$ new vertices $t+1, t+2, \ldots, t+\ell$ with weights $W_{t+1}, W_{t+2}, \ldots, W_{t+\ell}$ sampled independently from $\mu$ and the directed edges $(j, t+1), (j, t+2), \ldots, (j,t + \ell)$ oriented towards the newly arriving vertices. We say that $j$ is the \emph{parent} of the new-coming vertices. 
\end{enumerate}
Note that, since $\ell$ new vertices are connected to a parent at each time-step, for any vertex $i$ in the tree, $\ell$ divides the out-degree of $i$. Moreover, the evolution of the out-degree of vertex $i$ with weight $W_i$ is determined by the values $(f(j, W_i))_{j \in \N_0}$. 
In general, when the distribution $\mu$, fitness function $f$ and $\ell$ are specified, we refer to this model as a $(\mu,f, \ell)$-\textit{recursive tree with independent fitnesses}, often abbreviated as a ``$\rif{\mu, f,\ell}$'' for brevity. Here `independent fitnesses' refers to the fact that the fitness associated with a given vertex does not depend on the weights of its neighbours, in contrast to, for example, the models of dynamical \emph{simplicial complexes} studied in \cite{dynamical-simplices}. 

\redd{
\begin{rmq}
If we adopt the convention that the process terminates when no vertex can be chosen in the next step, the assumption that $f(0,W_0) > 0$ almost surely may be dropped in many places in this paper, if we condition on the event that number of vertices in the tree tends to infinity as $t \to \infty$.
\end{rmq}}
\bluee{
\begin{rmq} \label{rem:non-explicit-law}
In this model, the law of the sequence $(f(k,W))_{k\in \N_0}$ is more important than the function $f$. It is possible, for example, to define this model so that $(f(k,W))_{k\in \N_0}$ is any law on $\mathbb{R}_{+}^{\mathbb{N}_0}$ depending on $W$, even if one cannot write $f$ explicitly, and as long as the sequences associated with different vertices are independent. For example, the sequence could be any stochastic process indexed by the non-negative integers, depending on an initial source of randomness $W$.
\end{rmq}
}

\subsection{Quantities of Interest Studied in this Paper} \label{sec:quant-of-interest}
In this section, we will introduce the main quantities we will be interested in studying in this paper, along with some important definitions. Note that the definitions we introduce in this section will depend on the underlying parameters of the tree, $\mu$, $f$ and $\ell$. 

In this paper we will generally be concerned with the limiting behaviour of the following quantities:
\begin{enumerate}
    \item Given a Borel set $B \subseteq \mathbb{R}_{+}$, the quantity $N_{k}(t,B)$ denotes the number of vertices $v$ in the tree $\cT_t$ with out-degree $k \ell$ and weight $W_v \in B$, that is, 
    \begin{equation} \label{eq:deg-dist}
        N_{k}(t, B) := \sum_{v \in \mathcal{T}_{t}: \deg^{+}(v, \cT_{t}) = k\ell} \mathbf{1}_{B}(W_v).
    \end{equation}
    \item Given a Borel set $B \subseteq \mathbb{R}_{+}$, the quantity $\Xi(t,B)$ denotes the number of directed edges $(v,v')$ in the tree $\cT_t$ such that $W_v \in B$, that is, 
    \begin{equation}
 \label{eq:xi-def}
    \Xi(t,B) := \sum_{(v,v') \in \mathcal{T}_{t}} \mathbf{1}_{B}(W_v).
\end{equation}
\end{enumerate}

\bluee{Now, reasoning informally and non-rigorously for a moment, suppose that $W$ takes finitely many values. In addition, suppose that for all $t \geq t'$, where one considers $t'$ to be a `large' constant,  we have $\mathcal{Z}_t = \alpha t$, and for all $k \in \mathbb{N}_{0}$ we have $N_{k}(t, \{w\}) = \E{N_{k}(t, \{w\})} = \ell t \cdot n_{k}(\{w\})$ for some value $n_{k}(\{w\})$. The latter assumptions are motivated by the intuition that the respective quantities obey strong laws of large numbers. Then, for $t \geq t'$ and $k \geq 1$ we have
\begin{linenomath*}
\begin{align*}
    \ell n_{k}(\{w\}) & = \E{\E{N_{k}(t+1, \{w\}) - N_{k}(t, \{w\})|\mathcal{T}_{t}}}
    \\ &= \Prob{\text{vertex of out-degree $k$ and weight $w$ chosen}} \\& \hspace{2cm}- \Prob{\text{vertex of out-degree $k-1$ and weight $w$ chosen}}
    \\ &= N_{k-1}(t, \{w\}) \cdot \frac{f(k-1, w)}{\cZ_t} - N_{k}(t, \{w\}) \cdot \frac{f(k,w)}{\cZ_t}
    \\ &= \frac{\ell n_{k-1}(\{w\})\cdot f(k-1,w)}{\alpha} - \frac{\ell n_{k}(\{w\})\cdot f(k,w)}{\alpha}.
\end{align*}
\end{linenomath*}
Meanwhile, for $k =0$ we have
\begin{linenomath*}
\begin{align*}
    \ell n_{0}(\{w\}) & = \E{\E{N_{0}(t+1, \{w\}) - N_{0}(t, \{w\})|\mathcal{T}_{t}}}
    \\ &= \Prob{\text{newly arriving vertex with weight $w$}} \\ & \hspace{2cm} - \Prob{\text{vertex of out-degree $0$ and weight $w$ chosen}}
    \\ &= \ell \mu(\{w\}) - N_{0}(t, \{w\}) \cdot \frac{f(0,w)}{\cZ_t}
    \\ &= \ell \mu(\{w\}) - \frac{\ell n_{0}(\{w\}) \cdot f(0,w)}{\alpha}.
\end{align*}
\end{linenomath*}
Solving the recursion from the above two displays, for all $k \in \N_0$ we have
\begin{linenomath*}
\begin{align*}
    n_{k}(\{w\}) & = \mu(\{w\}) \cdot \frac{\alpha}{f(k, w) + \alpha}\prod_{i=0}^{k-1} \frac{f(i,w)}{f(i,w) + \alpha}
    \\ & = \E{\frac{\alpha}{f(k, W) + \alpha}\prod_{i=0}^{k-1} \frac{f(i,W)}{f(i,W) + \alpha}\mathbf{1}_{\{w\}}(W)}.
\end{align*}
\end{linenomath*}
}
It is therefore reasonable to expect that the limit of $\frac{N_{k}(t,B)}{\ell t}$ belongs to \bluee{a one parameter family $p^{\lambda}_{k}(\cdot)$ indexed by a positive real number $\lambda$ such that}
\begin{equation} \label{eq:def-1-param-family}
p^{\lambda}_{k}(B) := \E{\frac{\lambda}{f(k,W) + \lambda} \prod_{i=0}^{k-1} \frac{f(i,W)}{f(i,W) + \lambda} \mathbf{1}_{B}(W)},
\end{equation}
\bluee{where the parameter $\lambda$ can be recovered by the asymptotics of the partition function, so that the limit satisfies $\lambda = \alpha > 0$. 
It is important to note that in this paper, it may be the case that the limit of $\frac{N_{k}(t,B)}{\ell t}$ belongs to this family, but we do not necessarily have a strong law for the asymptotics of the partition function. For example, this is the case if the conditions of Theorem~\ref{th:degree-dist} are satisfied, but not those of Theorem~\ref{th:strong-law-parti}. 
}\\
Now, note that for every $t \in \mathbb{N}_0$, by computing the number of directed edges $(v,v')$ in $\cT_{t}$ with $W_v \in B$ in two different ways, we have 
\begin{equation} \label{eq:xi-as-sum}
\Xi(t,B) = \sum_{k=0}^{t} \ell k N_{k}(t,B). 
\end{equation}
When we normalise by $\ell t$, if, for $k \in \mathbb{N}_0$ the limit of $\frac{N_{k}(t,B)}{\ell t}$ is $p^{\alpha}_{k}(B)$, by Fatou's lemma we get 
\begin{equation} \label{eq:edge-dist-fatou}
\liminf_{t \to \infty} \frac{\Xi(t,B)}{\ell t} \geq \sum_{k=0}^{\infty} \ell k p^{\alpha}_{k}(B),
\end{equation}
which motivates the definition of the following family \bluee{indexed by a positive real number $\lambda$, such that}:
\begin{equation} \label{eq:def-1-param-family-2}
m(\lambda,B) := \sum_{k=0}^{\infty} \ell k p^{\lambda}_{k}(B) = \ell \cdot \E{\sum_{n=1}^{\infty} \prod_{i=0}^{n-1} \frac{f(i,W)}{f(i,W) + \lambda} \mathbf{1}_{B}(W)}.
\end{equation}
Now, if the limit exists, since we add $\ell$ edges at each time-step, the limit of the measures $\Xi(t, \cdot)/\ell t$ are probability measures. However, if $m(\alpha,\cdot)$ is not a probability distribution,
we can show that there exists a measurable set $B$ such that 
\[
\limsup_{t \to \infty} \frac{\Xi(t,B)}{\ell t} > m(\lambda,B).
\]
In this case, the inequality in~\eqref{eq:edge-dist-fatou} is strict, so that, after normalising by $\ell t$, the operations of taking limits in $k$ and in $t$ in \eqref{eq:xi-as-sum} do not commute. Thus, the set $B$ has acquired additional ``mass'' in the limit, and \redd{thus this phenomenon may be regarded as \emph{condensation}}. In Section~\ref{subsec-cond} we derive an example of this in the case that $f(i,W) = g(W)i + h(W)$, where $g$ is \bluee{bounded}. This generalises the case $f(i,W) = (i+1)W$ which has already been studied in \cite{borgs-chayes,dereich-ortgiese, dereich-mailler-morters} under the name \emph{preferential attachment with multiplicative fitness}.

\subsection{Open Problems}
The discussion in Section~\ref{sec:quant-of-interest} shows that much of the analysis of this model depends on a parameter $\alpha$. We conjecture that, in general, this parameter makes $m(\lambda, \cdot)$ `as close as possible' to a probability distribution, so that
\begin{equation} \label{def:gen-malth}
\alpha =\inf{\left\{\lambda > 0 :m(\lambda, \mathbb{R}_{+}) \leq 1 \right\}} \qquad
\text{if $m(\lambda,  \mathbb{R}_{+}) < \infty$ for some $\lambda >0$},
\end{equation}
\bluee{where we follow the convention that $\inf{\varnothing} = \infty$}.
\begin{conj} \label{conj:gen-degrees}
 Let $\cT$ be a \rif{\mu, f, \ell}, with $\alpha$ as defined in \eqref{def:gen-malth}. Then, for each $k \in \N_0$ and measurable set $B$, almost surely, we have
\begin{equation}
\frac{N_{k}(t, B)}{\ell t} \xrightarrow{t \to \infty} 
\begin{cases}
p^{\alpha}_{k}(B), & \text{if $\alpha < \infty$,} \\ 
\mu(B)\mathbf{1}_{\{0\}}(k), & \text{otherwise}.
\end{cases}
\end{equation} 
\end{conj}
The conjectured limit in the case when $\alpha = \infty$ is obtained by taking the limit of $p^{\alpha}_{k}(B)$ as $\alpha \to \infty$. This limit is $0$ unless $k = 0$, in which case it is $\mu(B)$. 
\bluee{
\begin{rmq}
Conjecture~\ref{conj:gen-degrees} has a natural analogue in the setting, as in Remark~\ref{rem:non-explicit-law}, that the sequence $(f(k,W))_{k \in \mathbb{N}_{0}}$ is instead given by a general stochastic process indexed by $\mathbb{N}_{0}$, depending on an initial source of randomness $W$. In this case, the expectation in~\eqref{eq:def-1-param-family} is instead taken over evolution over all sequences, with initial $W \in B$.The techniques used in this paper translate without modification to this case; the only important feature being that the sequences corresponding to different vertices are independent of each other. 
\end{rmq}}
\redd{
\begin{rmq}
Its important to note the order of the quantifiers in Conjecture~\ref{conj:gen-degrees}: given a sequence $(\mu_{j})_{j \in \mathbb{N}}$ of random measures on $\mathbb{R}_{+}$, such that, for any measurable set $B$, $\lim_{j \to \infty} \mu_{j}(B) = \mu_{\infty}(B)$ almost surely; it not necessarily the case that, \emph{almost surely, for all measurable sets $B$} we have $\lim_{j \to \infty} \mu_{j}(B) = \mu_{\infty}(B)$. However, it is the case that almost surely, $\mu_{j} \rightarrow \mu_{\infty}$ in the weak topology. This uses the fact that there exists a countable family of measurable sets such that any open set in $\mathbb{R}_{+}$ may be expressed as a disjoint, countable union of elements of this family\footnote{For example, one may take the set of all \emph{Dyadic intervals}, with endpoints of the form $j \cdot 2^{-n}$, $(j+1) \cdot 2^{-n}$, where $j, n \in \N_{0}$.} and then an application of the Portmanteau theorem. This approach is also used in this paper in the proof Corollary~\ref{cor:cond-weak-conv}.  
\end{rmq}
}

The discussion in Section~\ref{sec:quant-of-interest} described the quantity $\alpha$ as being closely related to the partition function. As a result, we also conjecture:

\begin{conj} \label{conj:strong-law-parti}
 Let $\cT$ be a \rif{\mu, f, \ell}, with $\alpha$ as defined in \eqref{def:gen-malth}. Then we have
\[\frac{\cZ_{t}}{t} \xrightarrow{t \to \infty} \alpha, \quad \text{ almost surely.} \]
\end{conj}

\subsection{Important Technical Conditions and Overview of Results}
\label{sec:overview} 
In this paper, we make partial progress towards the proofs of Conjecture~\ref{conj:gen-degrees} and Conjecture~\ref{conj:strong-law-parti}. We will refer to the following technical conditions: \hypertarget{c1}{} 
\begin{itemize}
    \item [\hyperlink{c1}{\textbf{C1}}] With $m(\lambda, \cdot)$ as defined in \eqref{eq:def-1-param-family-2}, there exists some $\lambda > 0$ such that 
    \begin{equation} \label{eq:cond-c1}
    1 < m(\lambda, \mathbb{R}_{+}) < \infty.
    \end{equation}
    Under this condition, by monotonicity, there exists a unique $\alpha > 0$ such that $m(\alpha, \mathbb{R}_{+}) = 1$, we call this the \emph{Malthusian parameter} associated with the process. \hypertarget{c2}{}
    \item [\hyperlink{c2}{\textbf{C2}}] There exists $\alpha > 0$ such that 
    \begin{equation} \label{eq:cond-c2}
    \lim_{t \to \infty} \frac{\cZ_t}{t} = \alpha.
    \end{equation}
\end{itemize}
Note that in~\eqref{def:gen-malth}, Conditions~\hyperlink{c1}{\textbf{C1}} and \hyperlink{c2}{\textbf{C2}}, we use the same symbol $\alpha$. This is because we conjecture that these coincide in general. In general, as we only assume either \hyperlink{c1}{\textbf{C1}} or \hyperlink{c2}{\textbf{C2}} at a time, the definition will be clear from context.

The paper will be structured as follows: \\\\
{\bf Section~\ref{sec:cont-time}:} We analyse the model under Condition~\hyperlink{c1}{\textbf{C1}}. 
        \begin{itemize}
            \item  In Theorem~\ref{th:degree-dist} we prove Conjecture~\ref{conj:gen-degrees} under Condition~\hyperlink{c1}{\textbf{C1}}, and as a consequence, in Theorem~\ref{th:edge-asymptotics} we show that, \redd{for any measurable set $B$}, $\Xi(t,B)/\ell t$ converges \redd{almost surely} to $m(\alpha, B)$. 
            \item In Theorem~\ref{th:strong-law-parti} we derive a condition under which~\hyperlink{c1}{\textbf{C1}} implies~\hyperlink{c2}{\textbf{C2}}. In particular, this proves Conjecture~\ref{conj:strong-law-parti} under this condition and~\hyperlink{c1}{\textbf{C1}}.
            \item The approaches used in this section are well-established, applying classical results in the theory of \emph{Crump-Mode-Jagers branching processes}, in a similar manner to the approaches taken by the authors of \cite{rudas,holmgren-janson,bhamidi, dereich-mailler-morters}. Nevertheless, these theorems have novel applications: we apply these theorems to the evolving Cayley tree considered by Bianconi in Example~\ref{wct} and the \emph{weighted random recursive tree}.
        \end{itemize}
{\bf Section~\ref{sec:gpaf}:} We analyse a particular case of the model when the fitness function \bluee{is such that} $f(i,W) = g(W)i + h(W)$, which we call the \emph{generalised preferential attachment tree with fitness} ($\gpaf$-tree).
This extends the existing models of preferential attachment with additive fitness, i.e., $f(i,W) = i + 1 + W$, and multiplicative fitness, i.e., $f(i,W) = (i+1)W$. When the functions $g$ \bluee{is} non-decreasing, we also treat the cases where Condition~\hyperlink{c1}{\textbf{C1}} can fail.  Let $\alpha$ be as defined in~\eqref{def:gen-malth}, and also define $\Lambda := \left\{\lambda > 0: m(\lambda, \mathbb{R}_{+}) < \infty \right\}$.
        \begin{itemize}
            \item  We consider the situation in which Condition~\hyperlink{c1}{\textbf{C1}} fails by having $m(\lambda,\mathbb{R}_{+}) \leq 1$ for all $\lambda \in \Lambda$. In this case, $m(\lambda, \mathbb{R}_{+})$ \bluee{is finite} for some $\lambda > 0$, but never exceeds $1$, so that $m(\alpha, \mathbb{R}_{+}) \leq 1$. In Theorem~\ref{thm:cond-edge-dist} we prove Conjecture~\ref{conj:gen-degrees} and Conjecture~\ref{conj:strong-law-parti} in this case, showing, in particular, that if $m(\alpha, \mathbb{R}_{+}) < 1$ the $\gpaf$-tree exhibits a \emph{condensation} phenomenon.
            \item Alternatively, Condition~\hyperlink{c1}{\textbf{C1}} may fail by having $\alpha = \infty$. Theorem~\ref{th:conn-transitions} also confirms Conjecture~\ref{conj:gen-degrees} in this case, 
            showing that the limiting degree distribution is \emph{degenerate}: almost surely the proportion of leaves in the tree tends to $1$. Moreover, we show that the \emph{fittest take all} of the mass of the distribution of edges according to weight, in the sense that
            \bluee{a proportion of edges tending to $1$ accumulates around the sets of vertices with weights conferring higher and higher fitness.}
            \item The techniques in this section are inspired by the coupling techniques exploited in \cite{borgs-chayes} and \cite{dereich-mailler-morters}, and extend the well known phase transition associated with the model of preferential attachment with multiplicative fitnesses studied in \cite{borgs-chayes,dereich-ortgiese,dereich-mailler-morters}. This generalisation shows that the phase-transition depends on the parameter $h$ too, so that, in some circumstances, condensation occurs, but vanishes if $h$ is increased enough pointwise (see Section~\ref{subsub:gpaf-cond-implic}). This is interesting because $h(W)$ may be interpreted as the `initial' popularity of a vertex when it arrives in the tree, showing that in order for the condensation to occur, there needs to be sufficiently many vertices of `low enough' initial popularity.  
            As far as the author is aware, these results are not only novel in the mathematical literature, but also in the general scientific literature concerning complex networks. 
        \end{itemize}
    {\bf Section~\ref{sec:gen-conv}:} We analyse the model under Condition~\hyperlink{c2}{\textbf{C2}}, proving general results for the distribution of vertices with a given degree and weight. 
     \begin{itemize}
         \item If the term $\alpha$ in Condition~\hyperlink{c2}{\textbf{C2}} is finite, Theorem~\ref{th:general-convergence} and Theorem~\ref{th:deg-degs2} confirm a weaker analogue of Conjecture~\ref{conj:gen-degrees} under this condition.
         \item The techniques used in this section are similar to those used in the proof of  \cite[Theorem~6]{dynamical-simplices}, however, in this instance we present a considerably shorter proof.
     \end{itemize}

\section{Analysis of \texorpdfstring{$\rifs{\mu, f, \ell}$}{} assuming \texorpdfstring{\protect\hyperlink{c1}{{\textbf{C1}}}}{}}
\label{sec:cont-time}
In order to apply Condition~\hyperlink{c1}{{\textbf{C1}}} in this section, we study a branching processes with a \emph{family tree} made up of individuals and their offspring whose distribution is 
identical to the discrete time model at the times of the branching events. In Section~\ref{subsec:cont-time-desc}, we describe this continuous time model, state Theorem~\ref{th:degree-dist} and state and prove Theorem~\ref{th:edge-asymptotics}. In Section~\ref{subsection-cmj} we include the relevant theory of \emph{Crump-Mode-Jagers} branching processes and use this to prove Theorem~\ref{th:degree-dist}. In Section~\ref{subsec:slln-parti} we apply the same theory, along with some technical lemmas to state and prove a strong law of large numbers for the partition function in Theorem~\ref{th:strong-law-parti}. We conclude the section with some interesting examples in Section~\ref{subsec:examples}. 

\subsection{Description of Continuous Time Embedding} \label{subsec:cont-time-desc}
In the continuous time approach, we begin with a population consisting of a single vertex $0$ with weight $W_0$ sampled from $\mu$ and an associated exponential clock with parameter $f(0,W_0)$. Then recursively, when the $i$th birth event occurs in the population, with the ringing of an exponential clock associated to vertex $j$: 
\begin{enumerate}[(i)]
\item Vertex $j$ produces offspring $\ell (i-1) + 1, \ldots, \ell i$ with independent weights $W_{\ell (i-1) +1}, \ldots, W_{\ell i}$ sampled from $\mu$ and exponential clocks with parameters $f(0, W_{\ell (i-1) +1}), \ldots, f(0, W_{\ell i})$. 
\item Suppose the number of offspring of $j$ before the birth event was $m$, so that its out-degree in the family tree is $m$. Then, the exponential random variable associated with $j$ is updated to have rate $f(m/\ell + 1, \bluee{W_{j}})$. If $f(m/\ell + 1, \bluee{W_{j}}) = 0$, then $j$ ceases to produce offspring and we say $j$ has \textit{died}.
\end{enumerate}
Now, if we let $\cZ_{i-1}$ denote the sum of rates of the exponential clocks in the population when the population has size $i-1$, the probability that the clock associated with $j$ is the first to ring is $f(m/\ell, W_j)/\cZ_{i-1}$. Hence, the family tree of the continuous time model at the times of the birth events $(\sigma_{i})_{i \geq 0}$ has the same distribution as the associated \rif{\mu, f, \ell}. The continuous time branching process is actually a Crump-Mode-Jagers branching process, which we will describe in more depth in Section~\ref{subsection-cmj}. 

To describe the evolution of the degree of a vertex in the continuous time model, we define the pure birth process with underlying probability space $(\Omega, \mathcal{F}, \mathbb{P})$ and state space $\ell \mathbb{N}$ as follows: first sample a weight $W$ and set $Y(0) = 0$. Let $\mathbb{P}_{w}$ denote the probability measure associated with the process when the weight sampled is $w$. Then, define the birth rates of $Y$ such that
\begin{equation} \label{def:yw}
    \Prob[w]{Y(t + h) = (k+1)\ell \; | \; Y(t) = k\ell} = f(k, w) h + o(h).
\end{equation} 
In other words, the time taken to jump from $k\ell$ to $(k + 1)\ell$ is exponentially distributed with parameter $f(k,w)$. 

Let $\rho$ 
denote the point process corresponding to the jumps in $Y$ and denote by $\E[w]{\rho(\cdot)}$ the intensity measure when the weight $W = w$. Also, denote by $\hat{\rho}_w$ the Laplace-Stieltjes transform, i.e., 
\[
\hat{\rho}_w(\lambda) : = \int_{0}^{\infty}e^{-\lambda t} \E[w]{\rho(\dd t)}.
\]
Note that, by Fubini's theorem, we have
\begin{linenomath*}
\begin{align} \label{eq:laplace2}
\hat{\rho}_w(\lambda) = \int_{0}^{\infty} \left(\int_{t}^{\infty} \lambda e^{- \lambda s} \dd s \right)\, \E[w]{\rho(\dd t)} & =\int_{0}^{\infty} \lambda e^{-\lambda s} \left(\int_{0}^{s} \E[w]{\rho(\dd t)}\right) \dd s \\ \nonumber & = \int_{0}^{\infty} \lambda e^{-\lambda s} \E[w]{Y(s)}\dd s.
\end{align}
\end{linenomath*}
Moreover, if we write $\tau_{k}$ for the time of the $k$th jump in $Y$, we have $\rho = \sum_{k=0}^{\infty} \ell \delta_{\tau_{k}}.
$
Note that, if the weight of $Y$ is $w$, \bluee{then} $\tau_{k}$ is distributed as a sum of independent exponentially distributed random variables with rates $f(0,w), f(1,w), \ldots, f(k-1,w)$, \bluee{where} we follow the convention that an exponential distributed random variable with rate $0$ is $\infty$.
Thus, we have that 
\begin{equation} \label{eq:laplace11}
\hat{\rho}_w(\lambda) = \ell \sum_{n=1}^{\infty} \E[w]{e^{-\lambda \tau_{n}}} = 
\ell\sum_{n=1}^{\infty} \prod_{i=0}^{n-1} \frac{f(i,w)}{f(i,w) + \lambda},
\end{equation} 
where in the last equality we have used the facts that a Laplace-Stieltjes transform of a convolution of measures is the product of Laplace-Stieltjes transforms and the Laplace-Stieltjes transform $\hat{X}(\lambda)$ of an exponential distributed random variable with parameter $s$ is $\int_{0}^{\infty} e^{-\lambda t} s e^{- s t} \dd t = \frac{s}{s+\lambda}$. Therefore, we see that $\E{\hat{\rho}_{W}(\lambda)} = m(\lambda, \mathbb{R}_{+})$ as defined in \eqref{eq:cond-c1}, and Condition~\hyperlink{c1}{\textbf{C1}} implies that there exists some $\lambda > 0$ such that $1 < \E{\hat{\rho}_{W}(\lambda)} < \infty$. In addition, the Malthusian parameter $\alpha$ appearing in Condition~\hyperlink{c1}{\textbf{C1}} is the unique \bluee{positive real number} such that 
\begin{equation} \label{eq:malth-def}
    \E{\hat{\rho}_{W}(\alpha)} = m(\alpha, \mathbb{R}_{+}) = \ell \cdot \E{\sum_{n=1}^{\infty} \prod_{i=0}^{n-1} \frac{f(i,W)}{f(i,W) + \alpha}} = 1.
\end{equation}
Our first result is the following: 
\begin{theorem}[Convergence of the Degree Distribution under \hyperlink{c1}{\textbf{C1}}] \label{th:degree-dist}
Let $\cT$ be a $\rif{\mu, f, \ell}$ satisfying~\hyperlink{c1}{\textbf{C1}} with Malthusian parameter $\alpha$. Then, \redd{for any measurable set $B \subseteq \mathbb{R}_+$}, with $N_{k}(t, B)$ as defined in \eqref{eq:deg-dist} and $p^{\alpha}_{k}(B)$ as defined in \eqref{eq:def-1-param-family}, we have
\begin{equation}
\frac{N_{k}(t, B)}{\ell t} \xrightarrow{t \to \infty}  p^{\alpha}_{k}(B),
\end{equation} 
almost surely.
\end{theorem}
The limiting formula for Theorem~\ref{th:degree-dist} has appeared in a number of contexts, and generalises many known results. Under Condition~\hyperlink{c1}{\textbf{C1}} this result was proved by Rudas, T\'{o}th and Valk\'{o} \cite{rudas} in the case that $W$ is constant and $\ell = 1$. The cases $f(i,W) = W(i+1)$ and $f(i,W) = i+1+W$ with $\ell = 1$ correspond respectively to the preferential attachment models with multiplicative and additive fitness mentioned in the introduction. In the multiplicative model, the result was first proved in \cite{borgs-chayes} and later in \cite{bhamidi}. In \cite{bhamidi}, Bhamidi also first proved the result for the case $f(i,W) = i + 1 + W$.  These models are examples of the generalised preferential attachment tree with fitness, which we study in more depth in Section~\ref{sec:gpaf}. Finally, the case $f(i,W) = W$, $\ell = 1$ corresponds to a model of weighted random recursive trees (see Example~\ref{ex:wrrt}). We postpone the proof of Theorem~\ref{th:degree-dist} to the end of Section~\ref{subsection-cmj}.

\begin{rmq}
 The limiting value has an interesting interpretation as a generalised geometric distribution. Consider an experiment where $W$ is sampled from $\mu$ and, given $W$, coins are flipped, where the probability of heads in the $i$th coin flip is proportional to $f(i,W)$ and tails proportional to $\alpha$. Then, the limiting distribution in Theorem~\ref{th:degree-dist} is the distribution of first occurrence of tails. Note that, by \hyperlink{c1}{\textbf{C1}}, the probability of infinite sequences of heads is $0$. 
\end{rmq}

\begin{rmq}
Note that $Y(t) < \infty$ for all $t \geq 0$ almost surely if $\tau_{\infty} := \lim_{k \to \infty} \tau_{k} = \infty$ almost surely.
The latter is satisfied if there exists $\lambda > 0 $ such that for almost all $w$
\[\E[w]{e^{-\lambda \tau_\infty}} = \lim_{n \to \infty} \E[w]{e^{-\lambda \tau_n}} = \lim_{n \to \infty}\prod_{i=0}^{n} \frac{f(i, w)}{f(i,w) + \lambda} = 0,
\]
which is implied by \hyperlink{c1}{\textbf{C1}}. In the literature concerning pure-birth Markov chains, this property is known as non-explosivity.
\end{rmq}

\begin{rmq} \label{rem:robustness}
In this paper, we have considered the case where the function $f$, and thus the birth process $Y$ as defined in~\eqref{def:yw}, depends on a single random variable $W$ taking values in $\mathbb{R}_{+}$. However, there is no loss of generality in assuming the random variable $W$ takes values in an arbitrary measure space, so long as the function $f$ is measurable. In particular, we may consider the case where the weight is given by a vector $(W_1,W_2)$ where $W_1$ and $W_2$ are possibly correlated random variables. 
\end{rmq}

\redd{Now, recall the definitions of $\Xi(t, \cdot)$ from \eqref{eq:xi-def} and $m(\alpha, \cdot)$ from \eqref{eq:def-1-param-family-2}. In the case that $m(\alpha, \cdot)$ is a probability distribution, the almost sure convergence of $N_{k}(t,B)/\ell n$ to $p^{\alpha}_{k}(B)$ for any measurable set $B$ is enough to imply that for any measurable set $B$ the quantity $\Xi(t,B)$ converges almost surely to $m(\alpha, B)$. Note that this condition is weaker than directly assuming \hyperlink{c1}{\textbf{C1}}. In particular, we have the following.}
%

\begin{theorem} \label{th:edge-asymptotics}
Assume $\cT$ is a \rif{\mu,f, \ell} with limiting degree distribution of the form $(p^{\alpha}_{k}(\cdot))_{k \in \mathbb{N}_{0}}$ and such that $m(\alpha, \mathbb{R}_{+}) = 1$. Then, for any measurable set $B$ we have 
\[
\frac{\Xi(t,B)}{\ell t} \xrightarrow{t \to \infty} m(\alpha, B),
\]
almost surely.
\end{theorem}
To prove this theorem,  we will apply the following elementary bound: for any two sequences $(a_n)_{n \in \mathbb{N}}, (b_{n})_{n \in \mathbb{N}}$,
such that either $\liminf_{n \to \infty} a_n > -\infty$ or $\limsup_{n \to \infty} b_n < \infty$, 
we have
\begin{equation} \label{eq:to-delete}
\liminf_{n \to \infty}(a_n + b_n) \leq \liminf_{n \to \infty} a_n + \limsup_{n \to \infty} b_n \leq \limsup_{n \to \infty}(a_n + b_n).
\end{equation}

\begin{proof}[Proof of Theorem~\ref{th:edge-asymptotics}]
Recall that, by \eqref{eq:xi-as-sum}, for each $t$, we have $\Xi(t,B) = \sum_{k=1}^{t} k \ell N_{k}(t,B)$. Also note that 
\begin{linenomath*}
\begin{align*}
\sum_{k=0}^{\infty} k \ell p^{\alpha}_{k}(B)
& = \ell \cdot \E{\left(\sum_{k=1}^{\infty} \frac{k\alpha}{f(k,W) + \alpha} \prod_{i=0}^{k-1} \frac{f(i,W)}{f(i,W) + \alpha} \right) \mathbf{1}_{B}(W)} \\ & = \ell \cdot \E{\left(\sum_{k=1}^{\infty} k \cdot \left(1 - \frac{f(k,W)}{f(k,W) + \alpha}\right) \prod_{i=0}^{k-1} \frac{f(i,W)}{f(i,W) + \alpha} \right) \mathbf{1}_{B}(W)}
\\ &  = \ell \cdot \E{\sum_{k=1}^{\infty}  \left(k\prod_{i=0}^{k-1} \frac{f(i,W)}{f(i,W) + \alpha} - k\prod_{i=0}^{k} \frac{f(i,W)}{f(i,W) + \alpha}\right)\mathbf{1}_{B}(W)} \\ & =
\ell \cdot \E{\left(\sum_{k=1}^{\infty} \prod_{i=0}^{k-1} \frac{f(i,W)}{f(i,W) + \alpha}\right) \mathbf{1}_{B}(W)}  = m(\alpha, B),
\end{align*} 
\end{linenomath*}
where the second to last equality follows from the telescoping nature of the sum inside the expectation. 
Thus, by Fatou's lemma, almost surely we have
\begin{equation} \label{eq:infbound}
m(\alpha, B) = \sum_{k=0}^{\infty} k \ell p^{\alpha}_{k}(B) = \sum_{k=0}^{\infty} k \ell \liminf_{t \to \infty} \frac{N_{k}(t,B)}{\ell t} \leq \liminf_{t \to \infty} \frac{\Xi(t,B)}{\ell t};
\end{equation}
and likewise, almost surely, $\liminf_{t \to \infty} \frac{\Xi(t,B^c)}{\ell t} \geq m(\alpha, B^c)$.
Now, since we add $\ell$ edges at every time-step, $\Xi(t,\mathbb{R}_{+}) = \ell t$. Thus, by~\eqref{eq:to-delete} 
\begin{linenomath*}
\begin{align*}
1 = \liminf_{t \to \infty} \left(\frac{\Xi(t, B)}{\ell t} + \frac{\Xi(t,B^{c})}{\ell t}\right)
& \leq \liminf_{t \to \infty} \frac{\Xi(t, B^c)}{\ell t} + \limsup_{t \to \infty} \frac{\Xi(t,B)}{\ell t} \\ & \leq \limsup_{t \to \infty} \left(\frac{\Xi(t, B)}{\ell t} + \frac{\Xi(t,B^{c})}{\ell t}\right) = 1.
\end{align*}
\end{linenomath*}
But, $m(\alpha, \cdot)$ is a probability measure, this is only possible if 
\begin{equation} \label{eq:supbound}
\liminf_{t \to \infty} \frac{\Xi(t, B^c)}{\ell t} = m(\alpha, B^{c}) \text{ and } \limsup_{t \to \infty} \frac{\Xi(t,B)}{\ell t} = m(\alpha, B) \text{ almost surely.} 
\end{equation}
Combining~\eqref{eq:infbound} and \eqref{eq:supbound} completes the proof.
\end{proof}

\subsection{Crump-Mode-Jagers Branching Processes}
\label{subsection-cmj}
In the continuous time setting, it is convenient to not only identify individuals of the branching process according to the order they were born, but also record their lineage, in such a way that the labelling encodes the structure of the tree. Therefore we also identify individuals of the branching process with elements of the infinite \emph{Ulam-Harris} tree $\mathcal{U} : = \bigcup_{n \geq 0} \N^{n}$, where $\N^{0} = \varnothing$ is the \emph{root}. In this case, an individual $u = u_1u_2\ldots u_k$ is to be interpreted recursively as the $u_k$th child of $u_1 \ldots u_{k-1}$. For example, $1, 2, \ldots$ represent the offspring of $\varnothing$. 

In \emph{Crump-Mode-Jagers (CMJ)} branching processes, individuals $u \in \mathcal{U}$ are equipped with independent copies of a random point process $\xi$ on $\mathbb{R}_+$. The point process $\xi$ associates \textit{birth times} to the offspring of a given individual, and we also may assume that $\xi$ has some dependence on a random weight $W$ associated with that individual. The process, together with birth times may be regarded as a random variable in the probability space $(\Omega, \Sigma, \mathbb{P}) = \prod_{x \in \mathcal{U}} (\Omega_{x}, \Sigma_{x}, \mathbb{P}_{x})$ where each $(\Omega_{x}, \Sigma_{x}, \mathcal{P}_{x})$ is a probability space with $(\xi_{x}, W_x)$ having the same distribution as $(\xi,W)$. We denote by $(\sigma_{i}^{x})_{i \in \N}$ points ordered in the point process $\xi_{x}$ and, for brevity, assume that $\xi(\{0\}) = 0$. We also drop the superscript when referring to the point process associated to $\varnothing$, so that $\sigma_{i} := \sigma^{\varnothing}_i$. Now, we set $\sigma_{\varnothing} : = 0$ and recursively, for $x \in \mathcal{U}$, $\sigma_{xi} := \sigma_{x} + \sigma^{x}_{i}$. Finally, we set $\mathbb{T}_{t} = \{x \in \mathcal{U}: \sigma_{x} \leq t\}$ and note that for each $t \geq 0$, $\mathbb{T}_{t}$ may be identified with the \emph{family tree} of the process in the natural way.
Informally, $\mathbb{T}_{t}$ can be described as follows: at time zero, there is one vertex $\varnothing$, which reproduces according to $(\xi_{\varnothing}, W_{\varnothing})$. Thereafter, at times corresponding to points in $\xi_{\varnothing}$, descendants of $\varnothing$ are formed, which in turn produce offspring according to the same law. A crucial aspect of the study of CMJ processes are \emph{characteristics} $\phi_{x}$ associated to each element $x \in \mathcal{U}$. For $x \in \mathcal{U}$, let $\mathcal{U}_{x} := \left\{xu: u \in \mathcal{U}\right\}$. Then, the processes $\phi_{x}$ are identically distributed, non-negative stochastic processes on the space $(\Omega, \Sigma, \mathbb{P})$ associated with individuals $x$, which may depend on $(\xi_{z}, W_{z})_{z \in \mathcal{U}_{x}}$. Intuitively, these are processes that track `characteristics' not only of the individual $x$, but on its potential offspring $\{xy: y \in \mathcal{U}\}$.
We then define the \emph{general branching process counted with characteristic} as
\[
Z^{\phi}(t) := \sum_{x \in \mathcal{U}: \sigma_x \leq t} \phi_x(t-\sigma_x);
\]
thus this function keeps a `score' of characteristics of individuals in the family tree associated with the process up to time $t$.
Let $\nu$ be the intensity measure of $\xi$, that is,  $\nu(B) := \E{\xi(B)}$ for measurable sets $B \subseteq \mathbb{R}_{+}$. 
A crucial parameter in the study of CMJ processes is the \emph{Malthusian parameter} $\alpha$ defined as the solution (if it exists) of
\[
\E{\int_{0}^{\infty} e^{-\alpha u} \xi(\dd u)} = 1.\]
Assume that $\nu$ is not supported on any lattice, i.e., for any $h > 0$ $\Supp{\nu} \subsetneq \{ 0, h, 2h, \ldots\}$, and that the first moment of $e^{-\alpha u} \nu(\dd u)$ is finite, i.e., $\int_{0}^{\infty} u e^{-\alpha u} \nu(\dd u) < \infty$. Nerman \cite{nerman_81} proved the following theorem.  
\begin{theorem}[{\cite[Theorem~6.3]{nerman_81}}] \label{th:Nerman}
Suppose that there  exists $\lambda < \alpha$ satisfying
\begin{equation} \label{cond1}
\E{\int_{0}^{\infty} e^{-\lambda s} \xi(\dd s)} < \infty. 
\end{equation}
Then, for any two c\`adl\`ag characteristics $\phi^{(1)}, \phi^{(2)}$ such that $\E{\sup_{t \geq 0} e^{-\lambda t} \phi^{(i)}(t)} < \infty, \; i=1, 2,$
we have 
\[ \lim_{t \to \infty} \frac{Z^{\phi^{(1)}}(t)}{Z^{\phi^{(2)}}(t)} = \frac{\int_0^\infty \mathrm e^{-\alpha s}\E{\phi^{(1)}(s)} \dd s}{\int_0^\infty \mathrm e^{-\alpha s}\E{\phi^{(2)}(s)} \dd s}, \]
almost surely on the event $\{\left|\mathbb{T}_{t}\right| \rightarrow \infty\}$. 
\end{theorem} 
Recall the definition of $\rho$ as the point process associated with the jumps in the process $Y$ defined in \eqref{def:yw}.
Then, the continuous time model outlined in Section~\ref{subsec:cont-time-desc} is a CMJ process having $\rho$ as its associated random point process and weight $W$. In this case, the Malthusian parameter is given by $\alpha$ in \eqref{eq:malth-def} and moreover, Condition~\hyperlink{c1}{\textbf{C1}} implies that the first moment $\int_{0}^{\infty} t e^{-\alpha t} \hat{\rho}_{\mu}(\dd t)$ \bluee{is finite}.

Theorem~\ref{th:degree-dist} is now an immediate application of Theorem~\ref{th:Nerman}.
\begin{proof}[Proof of Theorem~\ref{th:degree-dist}]
\sloppy Consider the continuous time branching process outlined in Section~\ref{subsec:cont-time-desc} and denote by $\sigma'_1 < \sigma'_2 \cdots$ the times of births of individuals in the process. Then, $\cT_n$ has the same distribution as the family tree $\mathbb{T}_{\sigma'_n}$. 
For any measurable set $B \subseteq \mathbb{R}$, define the characteristics $\phi^{(1)}(t) = \mathbf{1}_{\left\{Y(t) = k\ell, W \in B\right\}}$ and $\phi^{(2)}(t) = \mathbf{1}_{\{t \geq 0\}}$, where $W$ denotes the weight of the process $Y$. Note that, $Z^{\phi^{(1)}}(t)$ is the number of individuals with $k\ell$ offspring and weight belonging to $B$ up to time $t$, while $Z^{\phi^{(2)}}(t) = |\mathbb{T}_{t}|$. 
Thus, 
\[\lim_{t \to \infty} \frac{Z^{\phi^{(1)}}(t)}{Z^{\phi^{(2)}}(t)} = \lim_{t \to \infty} \frac{N_{k}(t, B)}{\ell t}.\] 
Note that both $\phi^{(1)}(t)$ and $\phi^{(2)}(t)$ are c\`adl\`ag and bounded and moreover, Condition~\hyperlink{c1}{\textbf{C1}} implies that \eqref{cond1} is satisfied. Moreover, the assumption that $f(0,W) > 0$ almost surely implies that $\left|\mathbb{T}_{t}\right| \rightarrow \infty$ almost surely. Thus, by applying Theorem~\ref{th:Nerman}, 
\begin{equation} \label{crump1}
\lim_{t \to \infty} \frac{Z^{\phi^{(1)}}(t)}{Z^{\phi^{(2)}}(t)} = \alpha \int_0^\infty \mathrm e^{-\alpha s}\E{\mathbf{1}_{\left\{Y(s) = k\ell, W \in B\right\}}} \dd s = \E{\E[W]{ \left(e^{-\alpha \tau_{k}} - e^{-\alpha\tau_{k+1}}\right)}\mathbf{1}_{B}(W)}
\end{equation} 
where the last equality follows from Fubini's theorem and we recall that $\tau_{k}$ is the time of the $k$th event in the process $Y_{W}(t)$. Now, since, when $W = w$, $\tau_k$ is distributed as a sum of independent exponentially distributed random variables with rates $f(0,w), f(1,w) \ldots$, we have 
\begin{equation} \label{crump2}
\E{\E[W]{e^{-\alpha \tau_{k}}}\mathbf{1}_{B}(W)} =  \E{\left(\prod_{i=0}^{k-1} \frac{f(i,W)}{f(i,W) + \alpha} \right)\mathbf{1}_{B}(W)}.
\end{equation}
The result follows from combining \eqref{crump1} and \eqref{crump2}.
\end{proof}

\begin{rmq}
As noted by the authors of \cite{rudas}, Theorem~\ref{th:Nerman} can be applied to deduce a number of other properties of the tree, in particular the analogue of \cite[Theorem~1]{rudas} applies in this case as well. 
\end{rmq}

\subsection{A Strong Law for the Partition Function} \label{subsec:slln-parti}
 We can also apply Theorem~\ref{th:Nerman} to show that the Malthusian parameter $\alpha$ emerges as the almost sure limit of the partition function, under certain conditions on the fitness function $f$.

\begin{theorem} \label{th:strong-law-parti}
Let $(\mathcal{T}_{t})_{t \geq 0}$ be a $\rif{\mu, f, \ell}$ satisfying \hyperlink{c1}{\textbf{C1}} with Malthusian parameter $\alpha$. Moreover, assume that there exists a constant $C < \alpha$ and a non-negative function $\varphi$ with $\E{\varphi(W)} < \infty$ such that, for all $k \in \N_{0}$, $f(k,W) \leq C k + \varphi(W)$ almost surely. Then, almost surely
\[
    \frac{\parti_{t}}{t} \xrightarrow{t \to \infty} \alpha.
\]
\end{theorem}
In order to apply Theorem~\ref{th:Nerman}, we need to bound $\E{\sup_{t \geq 0} e^{-\lambda t} \phi^{(1)}(t)}$ for an appropriate choice of \bluee{$\lambda < \alpha$ and} characteristic $\phi^{(1)}$ that tracks the evolution of the partition function associated with the process. In order to do so, using the assumptions on $f(i,W)$, we will couple the process $Y$ defined in \eqref{def:yw} with an appropriate pure birth process $(\mathcal{Y}(t))_{t\geq 0}$ (Lemma~\ref{lem:pointless-coupling}) and apply Doob's maximal inequality to a martingale associated with $(\mathcal{Y}(t))_{t \geq 0}$ (Lemma~\ref{lem:martingale-lem}). \bluee{As we will see, our choice of $\lambda$ will be given by $C$, and this is the reason for the assumption that $C < \alpha$ in Theorem~\ref{th:strong-law-parti}.}

In order to define $\mathcal{Y}(t)$, first sample a weight $W$ and set $\mathcal{Y}(0) = 0$. 
 Then, if $\mathbb{P}_{w}$ denotes the probability measure associated with the process when the weight is $w$, define the rates such that
\begin{equation} \label{eq:yule-def}
\Prob[w]{\mathcal{Y}(t + h) = k + 1 \; | \; \mathcal{Y}(t) = k} = (C k + \varphi(w)) h + o(h). 
\end{equation}
We also let $\mathcal{Y}_{w}$ denote the process with the same transition rates, but deterministic weight $w$.

It will be beneficial to state a more general result, about pure birth processes $\left(\mathcal{X}(t)\right)_{t\geq 0}$ with linear rates, from the paper by Holmgren and Janson \cite{holmgren-janson}. For brevity, we adapt the notation and only include some specific statements from both theorems. 

\begin{lemma}[{\cite[Theorem~A.6 \& Theorem~A.7]{holmgren-janson}}] \label{lemma-appendix-janson}
Let $\left(\mathcal{X}(t)\right)_{t\geq 0}$ be a pure birth process with $\mathcal{X}(0) = x_0$ and rates such that 
\[\Prob{\mathcal{X}(t + h) = k + 1 \; | \; \mathcal{X}(t) = k} = (c_1 k + c_2) h + o(h),\]
for some constants $c_1, c_2 > 0$. Then, for each $t \geq 0$
\begin{equation} \label{eq:expected-growth}
\E{\mathcal{X}(t)} = \left(x_0 + \frac{c_2}{c_1}\right)e^{c_1 t} - \frac{c_2}{c_1}. 
\end{equation}
Moreover, if $x_0 = 0$ the probability generating function is given by
\begin{equation} \label{eq:pgf}
\E{z^{\mathcal{X}(t)}} = \left(\frac{e^{-c_1t}}{1-z\left(1 - e^{-c_1t}\right)}\right)^{c_2/c_1}.
\end{equation}
\end{lemma}

Finally, we will require Lemma~\ref{lem:martingale-lem} and Lemma~\ref{lem:pointless-coupling}. 
\begin{lemma} \label{lem:martingale-lem}
For any $w > 0$, the process $(e^{-C t}\left(\mathcal{Y}_{w}(t) + \varphi(w)/C\right))_{t \geq 0}$ is a martingale with respect to its natural filtration $(\mathcal{F}_t)_{t\geq 0}$.  Moreover,
\[\E{\sup_{t \geq 0} \left(e^{-C t}\mathcal{Y}(t) \right)} < \infty.\]
\end{lemma}

\begin{proof}
The process $(\mathcal{Y}_{w}(t))_{t \geq 0}$  is a pure birth process satisfying the assumptions of Lemma~\ref{lemma-appendix-janson}, with $c_1 = C$ and $c_2 = \varphi(w)$. Therefore, by \eqref{eq:expected-growth} and the Markov property, for any $t > s >0$ we have 
\[
\E{\mathcal{Y}_{w}(t) \; | \; \mathcal{F}_{s}} = \E{\mathcal{Y}_{w}(t) \; | \; \mathcal{Y}_{w}(s)} =
\left(\mathcal{Y}_{w}(s) + \frac{\varphi(w)}{C} \right)e^{C(t-s)} - \frac{\varphi(w)}{C},
\]
which implies the martingale statement.

Moreover, applying \eqref{eq:pgf} for the probability generating function, differentiating twice and evaluating at $z = 1$, we obtain
\[\E{\mathcal{Y}_{w}(t)\left(\mathcal{Y}_{w}(t) -1 \right)}
= \frac{\varphi(w)\left(C + \varphi(w)\right)}{C^2} \left(e^{C t} -1\right)^2,
\]
and thus 
after some manipulations, we find that for all $t \geq 0$
\[\E{e^{-2Ct}\left(\mathcal{Y}_{w}(t) + \varphi(w)/C\right)^2} \leq \redd{\frac{\varphi(w)^2}{C^2} + \frac{\varphi(w)}{C}\left(1 - e^{-Ct}\right)}.\]
Combining this $L^2$ quadratic bound with Doob's maximal inequality, we have
\begin{linenomath*}
\begin{align*}
\E{\sup_{t \geq 0} \left(e^{-C t}\mathcal{Y}_w(t) \right)} & \leq 
\E{\sup_{t \geq 0} \left(e^{-C t}\left(\mathcal{Y}_w(t) + \varphi(w)/C\right)\right)}
\\ & \leq A + B\varphi(w),
\end{align*} 
\end{linenomath*}
for constants $A, B$ depending only on $C$. 
Thus,
\[
    \E{\sup_{t \geq 0} \left(e^{-C t}\mathcal{Y}(t) \right)} = \E{\sup_{t \geq 0} \left(e^{-C t}\mathcal{Y}_{W}(t) \right)} \leq A + B\E{\varphi(W)} < \infty.
\]
\end{proof}

\begin{lemma} \label{lem:pointless-coupling}
Recall the definition of $Y$ in \eqref{def:yw} and assume that there exists a constant $C < \alpha$ and a non-negative function $\varphi$ with $\E{\varphi(W)} < \infty$ such that, for all $k \in \N_{0}$, $f(k,W) \leq C k + \varphi(W)$ almost surely. Then, there exists a coupling $(\hat{Y}(t),\hat{\mathcal{Y}}(t))_{t \geq 0}$ of $(Y(t))_{t \geq 0}$ and $(\mathcal{Y}(t))_{t \geq 0}$ such that, for all $t \geq 0$
\[\hat{Y}(t) \leq \ell \cdot \hat{\mathcal{Y}}(t). \]
\end{lemma}
In the following proof, we denote by $\exprv{r}$ the exponential distribution with parameter $r$.
\begin{proof}
First, we sample $\hat{W}$ from $\mu$ and use this as a common weight for $\hat{Y}$ and $\hat{\mathcal{Y}}$.
Now, let $\left(\varsigma_{i}\right)_{i \geq 0}$ be independent $\exprv{f(i, \hat{W})}$ distributed random variables. Then, for all $k > 0$ set  $\hat{\tau}_{k} = \sum_{i=0}^{k-1} \varsigma_{i}$ and \[\hat{Y}(t) = \sum_{k=1}^{\infty} k \ell \mathbf{1}_{\left\{\hat{\tau}_{k} \leq t < \hat{\tau}_{k+1}\right\}}.\] 

The $\varsigma_{i}$ can be interpreted as the intermittent time between jumps from state $i$ to $i+\ell$. 
 For all $t> 0$ construct the jump times of $(\hat{\mathcal{Y}}(t))_{t \geq 0}$ iteratively as follows: 
\begin{itemize}
    \item Note that by assumption $f(0, \hat{W}) \leq \varphi(\hat{W})$. Let $e_{0} \sim \exprv{\varphi(\hat{W}) - f(0,\hat{W})}$ and set $\varsigma'_{0} = \min\left\{e_{0}, \varsigma_{0}\right\}$. We may interpret $\varsigma'_{0}$ as the time for $\hat{\mathcal{Y}}$ to jump from $0$ to $1$.  
    \item Given $\varsigma'_{0}, \ldots, \varsigma'_{j}$, let
    $q_{j} := \sum_{i=0}^{j}\varsigma'_{i}$ and define
    $m_{j} := \hat{Y}(q_{j})/\ell$, i.e., the value of $\hat{Y}/\ell$ once $\hat{\mathcal{Y}}$ has reached $j+1$. Assume inductively that $m_{j} \leq j+1$ and set \[e_{j + 1} \sim \exprv{C(j + 1) +\varphi(\hat{W}) - f(m_{j}, \hat{W})} \quad \text{ and } \quad \varsigma'_{j+1} = \min\left\{e_{j}, \varsigma_{m_j} \right\}.\]
\end{itemize}
Observe that, since $\varsigma_{j+1}' \leq \varsigma_{m_j + 1}$, we have $m_{j+1} \leq j+2$, so we may iterate this procedure. 

It is clear that $(\hat{Y}(t))_{t \geq 0}$ is distributed like $(Y(t))_{t \geq 0}$ and using the properties of the exponential distribution, one readily confirms  that $(\hat{\mathcal{Y}}(t))_{t \geq 0}$ is distributed like $(\mathcal{Y}(t))_{t\geq 0}$. Finally, the desired inequality follows from the fact that $\hat{\mathcal{Y}}(t)$ always jumps before or at the same time as $\hat{Y}(t)$. 
\end{proof}

\begin{proof}[Proof of Theorem~\ref{th:strong-law-parti}]
Consider the continuous time embedding of the $\rif{\mu, f, \ell}$ and define the characteristics $\phi^{(1)}(t) : = \sum_{k=0}^{\infty} f(k,W) \mathbf{1}_{\{Y(t) = k\ell\}}$ and $\phi^{(2)}(t) : = \mathbf{1}_{\{t \geq 0\}}$. Recall that we denote by $(\tau_i)_{i \geq 1}$ the times of the jumps in $Y$ and that, for all $k \geq 0$, $f(k,W) \leq C k + \varphi(W)$ . Then, by Lemma~\ref{lem:pointless-coupling}, Lemma~\ref{lem:martingale-lem} and the assumptions of the theorem,  
\[
\E{\sup_{t \geq 0} \left(e^{-C t} \phi^{(1)}(t)\right)}  \stackrel{\text{Lem.}~\ref{lem:pointless-coupling}}{\leq}  \E{\sup_{t \geq 0} \left(e^{-C t} \left(C \mathcal{Y}_{W}(t) + \varphi(W)\right) \right)} \stackrel{\text{Lem.}~\ref{lem:martingale-lem}}{<} \infty 
 \]
Now, in this case $Z^{\phi^{(1)}}(t)$ is the total sum of fitnesses of individuals born up to time $t$, while $Z^{\phi^{(2)}}(t) = |\mathcal{T}_{t}|$. Thus, by Theorem~\ref{th:Nerman} and Fubini's theorem in the second equality, almost surely we have
\begin{linenomath*}
\begin{align} \label{eq:thm4-equation}
\lim_{n \to \infty} \frac{\mathcal{Z}_{n}}{\ell n} & = \alpha \int_{0}^{\infty} e^{- \alpha s} \E{\sum_{k=0}^{\infty} f(k,W) \mathbf{1}_{\{Y(s) = k\ell\}}} \dd s =\E{\sum_{k=0}^{\infty} f(k,W) \left(e^{-\alpha \tau_{k}} - e^{-\alpha \tau_{k+1}}\right)} \\ \nonumber & \hspace{3cm} =  \E{\sum_{k=1}^{\infty} \frac{\alpha f(k,W)}{f(k,W) + \alpha}\prod_{i=0}^{k-1} \frac{f(i,W)}{f(i,W) + \alpha}}.  
\end{align}
\end{linenomath*}
Now, recall that by \eqref{eq:malth-def} we have 
\[\E{\sum_{k=1}^{\infty} \frac{f(k,W)}{f(k,W) + \alpha}\prod_{i=0}^{k-1} \frac{f(i,W)}{f(i,W) + \alpha}} = \frac{1}{\ell},\]
and combining this with \eqref{eq:thm4-equation} proves the result.
\end{proof}

\subsection{Examples of Applications of Theorem~\ref{th:degree-dist}} \label{subsec:examples}
\subsubsection{Weighted Cayley Trees} \label{wct}
Consider the model where $f(k,W) = 0$ for $k \geq 1$ and $f(0,W) = g(W)$. Thus, at each step, a vertex with degree $0$ is chosen and produces $\ell$ children and thus this model produces an $(\ell + 1)$-\emph{Cayley} tree, i.e., a tree in which each node that is not a leaf has degree $\ell + 1$. Without loss of generality, by considering the pushforward of $\mu$ under $g$ if necessary, we may assume that $g(W) = W$. In this case, $\hat{\rho}_{\mu}(\lambda) = \ell \cdot \E{\frac{W}{W + \lambda}}$ and thus \hyperlink{c1}{\textbf{C1}} is satisfied as long as $\ell \geq 2$. Thus, $p^{\alpha}_{k}(B) = 0$ for all $k \geq 2$ and
\[
p_0(B) = \E{\frac{\alpha}{W + \alpha}\mathbf{1}_{B}(W)}, \quad p_1(B) = \E{\frac{W}{W + \alpha} \mathbf{1}_{B}(W)}.
\]
This rigorously confirms a result of Bianconi \cite{bianconi_cayley}. Note however, that in \cite{bianconi_cayley}, $\alpha$ is described as the almost sure limit of the partition function and we may only apply Theorem~\ref{th:strong-law-parti} under the assumption that $\E{W} < \infty$. 

In the notation of \cite{bianconi_cayley}, the weights $W$ are called `energies', using the symbol $\epsilon$, the function $g(\epsilon) := e^{\beta \epsilon}$, where $\beta > 0$ is a parameter of the model,  and $\alpha := e^{\beta \mu_{F}}$ is described as the limit of the partition function. Thus, the proportion of vertices with out-degree $0$ with `energy' belonging to some measurable set $B$ is
\[\E{\frac{1}{e^{\beta (\epsilon - \mu_{F})} + 1}\mathbf{1}_{B}(W)}, \]
which is known as a \emph{Fermi-Dirac} distribution in physics. 

\subsubsection{Weighted Random Recursive Trees} \label{ex:wrrt}
 In the case that $f(k,W) = W$, we obtain a model of \textit{weighted random recursive trees} with independent weights and 
    \hyperlink{c1}{\textbf{C1}} is satisfied with $\alpha = \E{W}$ provided $\E{W} < \infty$.
    Theorem~\ref{th:degree-dist} then implies that 
    \[
\frac{N_{k}(t, B)}{\ell t} \xrightarrow{t \to \infty} \E{\frac{\ell \E{W} W^k}{\left(W + \ell\E{W}\right)^{k+1}} \mathbf{1}_{B}(W)},
   \]
   almost surely.
This was observed in the case $\ell = 1$ by the authors of \cite[Proposition~3]{dynamical-simplices}. Note also that in this case Theorem~\ref{th:strong-law-parti} coincides with the usual strong law of large numbers.

The weighted random recursive tree has a natural generalisation to affine fitness functions. This is the topic of the next section. 

\section{Generalised Preferential Attachment Trees with Fitnesses}
\label{sec:gpaf}
In this section, we study $\rifs{\mu, f, \ell}$ in the specific case when the function $f$ takes an affine form, that is, $f(i,W) = ig(W) + h(W)$, for positive, measurable functions $g,h$. We call this particular case of the model a \emph{generalised preferential attachment tree with fitness} (which we abbreviate as a $\gpaf$-tree). The affine form of this model \greenn{means} that it is tractable to apply the coupling methods outlined in Section~\ref{subsub:gpaf-cond-coupling}, when Condition~\hyperlink{c1}{\textbf{C1}} fails. Moreover, this model is general enough to be an extension of not only the weighted random recursive tree, but also of the additive and multiplicative models studied in \cite{borgs-chayes, bhamidi}. 

Below, in Section~\ref{subsec:applying-c1-to-gpaf} we apply the theory of the previous section to this model when \hyperlink{c1}{\textbf{C1}} is satisfied. In the rest of Section~\ref{sec:gpaf}. In Section~\ref{subsec-cond}, we analyse the model when Condition~\hyperlink{c1}{\textbf{C1}} fails by having $m(\lambda, \mathbb{R}_{+}) \leq 1$ for all $\lambda > 0$ such that $m(\lambda, \mathbb{R}_{+}) < \infty$, stating and proving  Theorem~\ref{thm:cond-edge-dist}. Then, in Section~\ref{subsec:deg-degs-gpaf} we analyse the model when Condition~\hyperlink{c1}{\textbf{C1}} fails by having $m(\lambda, \mathbb{R}_{+}) = \infty$ for all $\lambda > 0$, stating and proving Theorem~\ref{th:conn-transitions}.

Note that in this section, we formulate our results in terms of functions $g$ and $h$ \bluee{of} a random variable $W$ taking values in $\mathbb{R}_{+}$. However, in the vein of Remark~\ref{rem:robustness}, we expect these results to extend to cases where $g$ and $h$ may depend on more general random variables. For example, there is no loss of generality in assuming $g$ and $h$ depend on possibly correlated random variables $W_1$ and $W_2$ assigned to a given vertex. In this case, the coupling technique applied in Section~\ref{subsub:gpaf-cond-coupling} needs to be adjusted accordingly, with appropriate ``truncations'' of the vector $(W_1,W_2)$.

\subsection{When the \texorpdfstring{$\gpaf$}{}-tree satisfies Condition~\texorpdfstring{\protect\hyperlink{c1}{\textbf{C1}}}{}} \label{subsec:applying-c1-to-gpaf}
In the context of the $\gpaf$-tree, Condition~\hyperlink{c1}{\textbf{C1}} states that there exists $\lambda > 0$ such that
\[ 
m(\lambda, \mathbb{R}_{+}) = \ell \cdot \E{\sum_{n=1}^{\infty} \prod_{i=0}^{n-1} \frac{g(W)i +h(W)}{g(W)i + h(W) + \lambda}} > 1. 
\]
First recall the definition of the birth process $Y$ from \eqref{def:yw} in Section~\ref{sec:cont-time}, with $f(k,\bluee{W}) = g(W)k + h(W)$. 
By applying \eqref{eq:expected-growth} from Lemma~\ref{lemma-appendix-janson} and the initial condition $Y(0) = 0$, for any $w \in \mathbb{R}_{+}$ we have
\[
\E[w]{Y(t)} = \left(\frac{h(w)}{g(w)}\right) e^{\ell g(w)t} - \frac{h(w)}{g(w)}.
\]
Now, \eqref{eq:laplace2} and \eqref{eq:laplace11} in Section~\ref{sec:cont-time} showed that  
\begin{equation} \label{gen-geom-sum}
\ell \cdot \sum_{n=1}^{\infty} \prod_{i=0}^{n-1} \frac{g(W)i +h(W)}{g(W)i + h(W) + \lambda} = \int_{0}^{\infty} \lambda e^{-\lambda s} \E[w]{Y(s)} \dd s = 
\begin{cases}
\frac{h(w)}{\lambda/\ell - g(w)} &  \text{if } \lambda/ \ell > g(w); \\
\infty & \text{otherwise}.
\end{cases}
\end{equation}
For a measurable function $g:\mathbb{R}_{+} \rightarrow \mathbb{R}_{+}$ we define $\esssup{(g)}$ such that 
\begin{equation} \label{eq:esssup-def}
    \esssup{(g)} := \inf\left\{a \in \mathbb{R}_{+}: \mu\left(\left\{x : \, g(x) > a\right\}\right) = 0\right\}.
\end{equation}
Therefore by \eqref{gen-geom-sum}, for $\lambda \geq \ell \cdot \esssup{(g)}$ we have $m(\lambda, \mathbb{R}_{+}) = \E{\frac{h(W)}{\lambda/\ell - g(W)}}$, while if $\lambda < \ell \cdot \esssup{(g)}$ we have $m(\lambda, \mathbb{R}_{+}) = \infty$. Thus, Condition~\hyperlink{c1}{\textbf{C1}} is satisfied if $\esssup{(g)} < \infty$, $\E{h(W)} < \infty$ and, for some $\lambda \geq \ell \cdot \esssup{(g)}$
\begin{equation} \label{eq:laplace-pref-fitness}
1 < \E{\frac{h(W)}{\lambda/\ell - g(W)}} < \infty.
\end{equation}
As a result, the Malthusian parameter $\alpha$ appearing in Condition~\hyperlink{c1}{\textbf{C1}} is given by the unique $\alpha > 0$ such that
\begin{equation} \label{eq:malth-gpaf}
    \E{\frac{h(W)}{\alpha/\ell - g(W)}} = 1.
\end{equation}
Note that the parameter $\ell$ in the model has the effect of re-scaling the Malthusian parameter $\alpha$. Also, since $\alpha \geq \ell \cdot \esssup(g)$, if $\E{h(W)} < \infty$, Theorem~\ref{th:strong-law-parti} applies and $\alpha$ may also be interpreted as the almost sure limit of the partition function associated with the process. Now, in the context of this model, the limiting value $p^{\alpha}_{k}(\cdot)$ from Theorem~\ref{th:degree-dist} is such that
\begin{equation} \label{eq:deg-dist-form-gpaf}
p^{\alpha}_{k}(B) = \E{\frac{\alpha}{g(W)k + h(W) + \alpha} \prod_{i=0}^{k-1} \frac{g(W)i + h(W)}{g(W)i + h(W) + \alpha} \mathbf{1}_{B}(W)}.
\end{equation}
Now, recall Stirling's approximation, which states that
\begin{equation} \label{eq:stirling_gamma_approx}
\Gamma(z) = \left(1 + O(1/z)\right)z^{z - \frac{1}{2}} e^{-z},
\end{equation}
\bluee{as $z \to \infty$}.
If $g(W) > 0$ on $B$, by dividing the numerator and denominator of terms inside the product in~\eqref{eq:deg-dist-form-gpaf}, we obtain a ratio of Gamma functions. Thus, by applying Stirling's approximation, on any measurable set $B$ on which $g,h$ are bounded,  we have 
\[p^{\alpha}_{k}(B) = \left(1 + O(1/k)\right)\E{c_{B}k^{-\left(1 + \frac{\alpha}{g(W)}\right)} \mathbf{1}_{B}(W)},\]
where $c_{B}$, which comes from the term outside the product in~\eqref{eq:deg-dist-form-gpaf}, depends on $g$ and $h$ but not $k$. Thus, the distribution of $(p^{\alpha}_{k}(B))_{k \in \mathbb{N}_{0}}$ follows what one might describe as an `averaged' power law. Moreover, in the case \bluee{that} $\ell = 1$, \bluee{we have} $\alpha \geq \esssup(g)$, thus, \[\E{c_{B}k^{-\left(1 + \frac{\alpha}{g(W)}\right)} \mathbf{1}_{B}(W)} \geq c' k^{-2}\]
for some $c' > 0$. It has been observed that real world complex networks, have power law degree distributions where the observed power law exponent lies between $2$ and $3$ (see, for example, \cite{vanderhofstad2016}). Note that by~\eqref{eq:malth-gpaf}, $\alpha$ depends on both $h$ and $g$, so that keeping $g$ fixed and making $h$ smaller has the effect of reducing the exponent of the power law.

In the remainder of this section we set $\ell =1$, for brevity. The arguments may be adapted in a similar manner to the case $\ell >1$.

\subsection{A Condensation Phenomenon in the \texorpdfstring{$\gpaf$}{}-tree when Condition~\texorpdfstring{\protect\hyperlink{c1}{\textbf{C1}}}{} Fails}  \label{subsec-cond}
 Recall that, in the $\gpaf$-tree, if $\lambda \geq \esssup{(g)}$ we have
 \begin{equation} \label{eq:m-lambda}
 m(\lambda, \mathbb{R}_{+}) = \E{\frac{h(W)}{\lambda - g(W)}},
 \end{equation}
 and if $\lambda < \esssup{(g)}$,  we have $m(\lambda, \mathbb{R}_{+}) = \infty$. If we define 
\begin{equation} \label{eq:big-lambda-def}
\Lambda := \left\{\lambda > 0: m(\lambda, \mathbb{R}_{+}) < \infty \right\},
\end{equation} 
in this subsection, we consider the case that the $\gpaf$-tree fails to satisfy Condition~\hyperlink{c1}{\textbf{C1}} by having $m(\lambda, \mathbb{R}_{+}) \leq 1$ for all $\lambda \in \Lambda$.
 We show that in this case the $\gpaf$-tree satisfies a formula for the degree distribution of the same form as \eqref{eq:def-1-param-family}. Moreover, if $\lambda^*:= \inf{\left(\Lambda\right)}$ and $m(\lambda^*, \mathbb{R}_{+}) < 1$, this model exhibits a condensation phenomenon, as described in Theorem~\ref{thm:cond-edge-dist}. We remark that such results have been proved for the case of the preferential attachment tree with multiplicative fitness, i.e., the case $h \equiv g$, in \cite{dereich-ortgiese}, in a more general framework; that is to say encompassing other models apart from a tree.

In Section~\ref{subsub:gpaf-cond-main} we state our main result, Theorem~\ref{thm:cond-edge-dist} and discuss interesting implications in Section~\ref{subsub:gpaf-cond-implic}. In Section~\ref{subsub:gpaf-cond-coupling} we state and prove Lemma~\ref{lem:coupl} which is the crucial tool used in proofs of the theorem. The proof of Theorem~\ref{thm:cond-edge-dist} is deferred to Section~\ref{subsub:proof-of-thm-cond}. 

\redd{
\begin{rmq} \label{rem:atom-max}
If $\cT$ does not satisfy $\hyperlink{c1}{\textbf{C1}}$ and $\E{h(W)} <  \infty$, we must have $\mu\left(\left\{x: g(x) = \esssup{g}\right\}\right) = 0$, since otherwise, for each $\lambda > \lambda^{*}$ we have $m(\lambda, \mathbb{R}_+) < \infty$, and by monotone convergence $\lim_{\lambda \downarrow \lambda^*} m(\lambda, \mathbb{R}_+) \uparrow \infty$.
\end{rmq}
}
\subsubsection{Theorem~\ref{thm:cond-edge-dist}: Condensation in the \texorpdfstring{$\gpaf$}{}-tree} \label{subsub:gpaf-cond-main}
\redd{Our main result in this subsection is the following theorem, which demonstrates the possibility of condensation in this model. Recall that in this section, we have $\lambda^* = \esssup{(g)}$. We then define the following family of sets of positive $\mu$-measure, depending on a parameter $\eps > 0$:
\begin{equation} \label{eq:def-condensate-sets}
    \mathcal{M}_{\eps} : = \left\{x: g(x) \geq \lambda^{*} - \eps\right\}.
\end{equation}}
\begin{theorem} \label{thm:cond-edge-dist}
Suppose $\cT = \left(\mathcal{T}_{t}\right)_{t \geq 0}$ is a $\gpaf$-tree, with associated functions $g$, where $g$ is bounded, $\E{h(W)} < \infty$, and Condition~\hyperlink{c1}{\textbf{C1}} fails. 
Then we have the following assertions:
\begin{enumerate}
\redd{\item For any measurable set $B$, such that for some $\eps > 0$ we have $B \subseteq \mathcal{M}^{c}_{\eps}$,
\begin{equation} \label{eq:setwise-corrected}
    \frac{\Xi(t,B)}{\ell t} \xrightarrow{t \to \infty} \E{\frac{h(W)}{\lambda^{*} - g(W)} \mathbf{1}_{B}(W)}, \quad \text{almost surely.}
\end{equation}
In particular,  if $\E{\frac{h(W)}{g(\wmax) - g(W)}} < 1$, for $\eps > 0$ sufficiently small, we have, \[\frac{\Xi(t,\mathcal{M}_{\eps})}{\ell t} \xrightarrow{t \to \infty} 1 - \E{\frac{h(W)}{\lambda^{*} - g(W)} \mathbf{1}_{\mathcal{M}^{c}_{\eps}}(W)} > \E{\frac{h(W)}{\lambda^{*} - g(W)} \mathbf{1}_{\mathcal{M}_{\eps}}(W)} = m(\lambda^*, \mathcal{M}_{\eps}),\] so that this model exhibits a condensation phenomenon, as described before Conjecture~\ref{conj:gen-degrees} in Section~\ref{sec:quant-of-interest}.
}
\item For any measurable set \redd{$B \subseteq \mathbb{R}_{+}$,} almost surely we have 
\begin{equation} \label{eq:deg-dist-cond}
\frac{N_{k}(t,B)}{t} \xrightarrow{t \to \infty} \E{\frac{\lambda^{*}}{g(W)k + h(W) + \lambda^{*}} \prod_{i=0}^{k-1} \frac{g(W)i + h(W)}{g(W)i + h(W) + \lambda^{*}} \mathbf{1}_{B}(W)} = p^{\lambda^{*}}_{k}(B).
\end{equation}
\item The partition function \bluee{satisfies} \[\frac{\cZ_{t}}{t} \xrightarrow{t \to \infty} \lambda^{*}, \quad \text{almost surely.} \]
\end{enumerate}
\end{theorem}

Assertion 1 of Theorem~\ref{thm:cond-edge-dist} leads to the following weak convergence result, \redd{if $g$ is increasing, and $\Supp{\mu} \subseteq [0,\wmax]$, where $\wmax := \sup{\left(\Supp{\mu}\right)} < \infty$. Define the measure $\pi(\cdot)$ such that, for any measurable set $B$, \[\pi(B) = \E{\frac{h(W)}{g(\wmax) - g(W)} \mathbf{1}_{B}(W)} + \left(1 - \E{\frac{h(W)}{g(\wmax) - g(W)}}\right)\delta_{\wmax}(B).\]}
\redd{
\begin{cor} \label{cor:cond-weak-conv}
Under the above assumptions, with respect to the weak topology, \[\frac{\Xi(t,\cdot)}{\ell t} \xrightarrow{t \to \infty} \pi(\cdot), \quad \text{almost surely.}\] 
\end{cor}
\begin{rmq}
Corollary~\ref{cor:cond-weak-conv} is the form in which \emph{condensation} results usually appear in the literature, showing that the limit of the sequence $\frac{\Xi(t,\cdot)}{\ell t}$ is no longer absolutely continuous with respect to $\mu$. In this regard, Corollary~\ref{cor:cond-weak-conv} generalises the case $f(i,W) = (i+1)W$ which has already been proved in \cite{borgs-chayes}.
\end{rmq}
\begin{proof}[Proof of Corollary~\ref{cor:cond-weak-conv}]
In view of the Portmanteau theorem, it suffices to prove that, almost surely, for any open set $O \subseteq [0, \wmax]$, we have 
\[
\liminf_{t \to \infty} \frac{\Xi(t, O)}{\ell t} \geq \pi(O). 
\]
Now, it is well-known that there exists a countable family of measurable sets $D_{1}, D_2, \ldots$ such that any open subset of $[0, \wmax]$ may be expressed as a countable disjoint union of elements of this family.\footnote{For example, one may take the set of all \emph{Dyadic intervals}, with endpoints of the form $j \cdot 2^{-n} \wmax$, $(j+1) \cdot 2^{-n} \wmax$, where $j, n \in \N_{0}$.} Fix such a family. Now, by Assertion 1 of Theorem~\ref{thm:cond-edge-dist}, it is the case that, almost surely, 
\begin{equation} \label{eq:countable-family-setwise}
\lim_{t \to \infty} \frac{\Xi(t, S)}{t} = \pi(S) \quad \forall S \in \mathcal{C},
\end{equation}
where $C$ is the countable collection of sets
\[
\mathcal{C} := \left\{D_{i} \cap \mathcal{M}^{c}_{1/j}, \, \mathcal{M}_{1/j}: \, i, j \in \mathbb{N}\right\}.
\]
Now, let $O$ be an arbitrary open set. First, suppose that $\wmax \notin O$. Then, for a pairwise disjoint collection $D_{i_1}, D_{i_2}, \ldots$ such that $O = \bigcup_{\ell \in \mathbb{N}} D_{i_\ell}$, for each $j, k \in \mathbb{N}$ we have 
\[
\liminf_{t \to \infty} \frac{\Xi(t,O)}{t} \geq \sum_{\ell = 1}^{k} \liminf_{t \to \infty} \frac{\Xi(t, D_{i_\ell} \cap \mathcal{M}^{c}_{1/j})}{t} \geq \sum_{\ell = 1}^{k} \pi(D_{i_\ell} \cap \mathcal{M}^{c}_{1/j}). 
\]
Taking limits in $j$ and $k$, the right hand side converges to $\pi(O)$, as required.
On the other hand, since $g$ is increasing, for each $\eps > 0$, there exists $\delta = \delta(\eps) > 0$ such that $\mathcal{M}_{\eps} \subseteq [\wmax - \delta, \wmax]$, and, moreover, $\delta \to 0$ as $\eps \to 0$. Therefore, because $O$ is open, for all $j$ sufficiently large, we have $\mathcal{M}_{1/j} \subseteq O$. But then, for a pairwise disjoint collection $D_{i_1}, D_{i_2}, \ldots$ such that $O = \bigcup_{\ell \in \mathbb{N}} D_{i_\ell}$ we have $O = \mathcal{M}_{1/j} \cup \left(\bigcup_{\ell \in \mathbb{N}} D_{i_\ell} \cap \mathcal{M}^{c}_{1/j}\right)$. Therefore, 
\begin{linenomath*}
\begin{align*}
\liminf_{t \to \infty} \frac{\Xi(t,O)}{t} &\geq \liminf_{t \to \infty} \frac{\Xi(t, \mathcal{M}_{1/j})}{t} + \sum_{\ell = 1}^{k} \liminf_{t \to \infty} \frac{\Xi(t, D_{i_\ell} \cap \mathcal{M}^{c}_{1/j})}{t} 
\\ & \geq \pi(\mathcal{M}_{1/j}) + \sum_{\ell = 1}^{k} \pi(D_{i_\ell} \cap \mathcal{M}^{c}_{1/j}), 
\end{align*}
\end{linenomath*}
so that, by again taking limits in $j$ and $k$, the right hand side converges to $\pi(O)$. The result follows. 
\end{proof}
}

\subsubsection{Some Interesting Implications of the Condensation Phenomenon}
\label{subsub:gpaf-cond-implic}
The condensation result in Theorem~\ref{thm:cond-edge-dist}  has interesting implications for the $\gpaf$-tree. Informally, the parameter $g(w)$ measures the extend to which the `popularity' of a vertex with weight $w$ is reinforced by the number of its neighbours, while the parameter $h(w)$ represents its `initial popularity'. The condensation phenomenon then depends on both $\mu$ and $h$, in the sense that condensation occurs if vertices of high weight are `rare enough' and the initial popularity is `low enough'. More precisely, 
\redd{if $\E{h(W)}, \lambda^{*} < \infty$}  we \redd{see that the tree displays the following interesting features:}
\begin{enumerate}
    \item By Remark~\ref{rem:atom-max}, if $\mu$ is such that $\E{\frac{1}{\lambda^* - g(W)}} = \infty$, Condition~\textbf{C1} is satisfied in this model, and thus, the model does not demonstrate a condensation phenomenon.
    \item Otherwise, if $g$ is such that $\E{\frac{1}{\lambda^* - g(W)}} = \redd{C^{'}} < \infty$, then either   
    \[
    \E{\frac{h(W)}{\lambda^* - g(W)}} > 1 \quad \text{ or } \quad \E{\frac{h(W)}{\lambda^* - g(W)}} \leq 1.
    \]
    In the first case, Condition~\textbf{C1} is satisfied, but fails in the second case. However, in the second case, if the inequality is strict, condensation arises. Therefore, for fixed $g$, condensation in this model arises by reducing $h$ sufficiently pointwise, for example, by replacing $h$ by $K \cdot h$ where $K < 1/\redd{C^{'}}$ is a constant. 
    \item \redd{It is the `reinforcement' rather that the `initial popularity' that drives the condensation, in the sense that, it may be the case, for example, that $h$ is minimised on the sets $\mathcal{M}_{\eps}$ where the condensation occurs.}
\end{enumerate}

\subsubsection{A Coupling Lemma}
\label{subsub:gpaf-cond-coupling}
In order to prove Theorem~\ref{thm:cond-edge-dist}, we first prove an additional lemma.
For each $\eps > 0$ such that \bluee{$\eps < \lambda^*$}, let $\cT^{+\eps} = (\cT^{+\eps}_{t})_{t \geq 0}$ and $\cT^{-\eps} = (\cT^{-\eps}_{t})_{t \geq 0}$ denote $\gpaf$-trees with the same \bluee{set of weights, but with the function $g$ modified to $g_{+\eps}$ and $g_{-\eps}$ respectively \bluee{on the set $\mathcal{M}_{\eps}$ from~\eqref{eq:def-condensate-sets}} such that 
\[
g_{+\eps} := g\mathbf{1}_{\mathcal{M}^{c}_{\eps}} + \lambda^*\mathbf{1}_{\mathcal{M}_{\eps}} \quad \text{and} \quad g_{-\eps} := g\mathbf{1}_{\mathcal{M}^{c}} + (\lambda^{*} - \eps)\mathbf{1}_{\mathcal{M}_{\eps}} 
\]}
\bluee{The motivation behind these choices of $\cT^{+\eps}$ and $\cT^{-\eps}$ is that, because $\cT$ does not satisfy Condition~\hyperlink{c1}{\textbf{C1}}, is that $g_{+\eps}$ and $g_{-\eps}$ attain their essential suprema on sets of positive measure}. Thus, because $\E{h(W)} < \infty$, 
by Remark~\ref{rem:atom-max} these \bluee{auxiliary} trees satisfy Condition~\hyperlink{c1}{\textbf{C1}}, and we may apply the theorems from Section~\ref{sec:cont-time} with regards to these trees. Then, \redd{using the fact that} these trees provide sufficiently good `approximations' of the tree $\cT$, we may deduce \redd{our} results by sending $\eps$ to $0$. 

In this vein, let $N^{+\eps}_{\geq k}(t,B), N_{\geq k}(t,B)$ and $N^{-\eps}_{\geq k}(t,B)$ denote the number of vertices with out-degree~$\geq k$ and weight belonging to the set $B$ in $\cT^{+\eps}_{t}, \cT_{t}$ and $\cT^{-\eps}_t$ respectively. In their respective trees, we also denote by 
$\cZ^{+\eps}_t, \cZ_t$ and  $\cZ^{-\eps}_t$ the partition functions at time $t$. Finally, for brevity, we write $f^{(+\eps)}_{t}(v), f_{t}(v)$ and $f^{(-\eps)}_{t}(v)$ for the fitness of a vertex $v$ at time $t$ in each of these models. \bluee{For example},~$f_{t}(v) = g(W_v)\outdeg{v, \cT_{t}} + h(W_v)$. 
\begin{lemma} \label{lem:coupl}
For all $\eps > 0$, there exists a coupling $(\hat{\cT}^{+\eps}, \hat{\cT}, \hat{\cT}^{-\eps})$ of these processes such that, on the coupling, for all $t \in \N_0$, 
\begin{enumerate}
\item \bluee{For all measurable sets $B  \subseteq \mathcal{M}_{\eps}^{c}$ we have $\Xi^{+\eps}(t,B) \leq \Xi(t,B) \leq \Xi^{-\eps}(t,B),$} 
\item For all measurable sets $B \subseteq \mathcal{M}_{\eps}^{c}$ and $k \in \N_0$, we have
\[N^{+\eps}_{\geq k}(t,B) \leq N_{\geq k} (t,B) \leq N^{-\eps}_{\geq k}(t,B);\]
\item \bluee{We have the inequalities} $\cZ^{-\eps}_{t} \leq \cZ_t \leq \cZ^{+\eps}_t.$
\end{enumerate}
\end{lemma}
\begin{proof}[Proof of Lemma~\ref{lem:coupl}]
\bluee{We construct the trees having the same sequence of weights $(W_{i})_{i \geq 0}$, so that the dynamics of the model are only influenced by differences in the function $g$ in the respective models. Thus, at time $0$ each of the models consist of single vertices labelled $0$ with weight $W_0$ and having fitness given by $h(W_0)$.}
Assume, that at the $t$th time-step,
\[(\hat{\cT}^{+\eps}_{n})_{0 \leq n \leq t} \sim (\cT^{+\eps}_{n})_{0 \leq n \leq t}, \quad (\hat{\cT}_{n})_{0 \leq n \leq t} \sim (\cT_{n})_{0 \leq n \leq t} \quad \text{and} \quad (\hat{\cT}^{-\eps}_{n})_{0 \leq n \leq t} \sim (\cT^{-\eps}_{n})_{0 \leq n \leq t}.\]
In addition, assume, by induction, that we have $\cZ^{-\eps}_{t} \leq \cZ_t \leq \cZ^{+\eps}_t$ and for each vertex $v$ with \bluee{$W_{v} \in \mathcal{M}^{c}_{\eps}$}
\begin{equation} \label{eq:mono-deg-couple}
\outdeg{v, \hat{\cT}^{+\eps}_t} \leq \outdeg{v, \hat{\cT}_{t}} \leq \outdeg{v, \hat{\cT}^{-\eps}_t}.
\end{equation} 
Note that~\eqref{eq:mono-deg-couple}, \bluee{and the fact that the trees have the same number of directed edges} imply the first and the second assertions of the lemma up to time $t$. As a result, for each vertex $v$ with \bluee{$W_{v} \in \mathcal{M}^{c}_\eps$} we have $f^{(+\eps)}_t(v) \leq f_t(v) \leq f^{(-\eps)}_{t}(v)$. Now, for the $(t+1)$st step 
\begin{itemize}
    \item Introduce a vertex $t+1$ with weight $W_{t+1}$ sampled independently from $\mu$. 
    \item Form $\hat{\cT}^{-\eps}_{t+1}$ by sampling the parent $v$ of $t+1$ independently according to the law of $\mathcal{T}^{-\eps}$, i.e., with probability proportional to $f^{(-\eps)}_t(v)$. Then, in order to form $\hat{\cT}_{t+1}$ sample an independent uniformly distributed random variables $U_1$ on $[0,1]$. 
    \begin{itemize}
        \item If $U_1 \leq \frac{\cZ^{-\eps}_{t}f_{t}(v) }{\cZ_t f^{(-\eps)}_{t}(v)}$ and \bluee{$W_{v} \in \mathcal{M}_{\eps}^{c}$}, select $v$ as the parent of $t+1$ in $\hat{\cT}_{t+1}$ as well.
        \item Otherwise, form $\hat{\cT}_{t+1}$ by selecting the parent $v'$ of $t+1$ with probability proportional to $f_{t}(v')$ out of all all the vertices with weight \bluee{$W_{v'} \in \mathcal{M}_{\eps}$}.
    \end{itemize}
    \item Then form $\hat{\cT}^{+\eps}_{t+1}$ \bluee{from $\hat{\cT}_{t+1}$ in an identical manner to the way $\hat{\cT}_{t+1}$ is formed from $\hat{\cT}^{-\eps}$, with another, independent uniform random variable $U_2$ on [0,1]}.
\end{itemize}
It is clear that $\hat{\cT}^{-\eps}_{t+1} \sim \cT^{-\eps}_{t+1}$. On the other hand, in $\hat{\cT}_{t+1}$ the probability of choosing a parent $v$
of $t+1$ with weight \bluee{$W_v \in \mathcal{M}_{\eps}^{c}$} is 
\[\frac{\cZ^{-\eps}_{t}f_{t}(v)}{\cZ_t f^{(-\eps)}_{t}(v)} \times \frac{f_t^{(-\eps)}(v)}{\cZ^{-\eps}_t} = \frac{f_t(v)}{\cZ_t},\]
whilst the probability of choosing a parent $v'$ with weight \bluee{$W_{v'} \in \mathcal{M}_{\eps}$} is 
\begin{linenomath*}
\begin{align*}
    & \frac{f_t(v')}{\sum_{v :W_{v} \geq \wmax - \eps} f_{t}(v)} \left(\sum_{v: W_{v} < \wmax - \eps}\left(1- \frac{\cZ^{-\eps}_{t}f_{t}(v)}{\cZ_t f^{(-\eps)}_{t}(v)}\right) \frac{f^{(-\eps)}_t(v)}{\cZ^{-\eps}_t}\right) \\ & \hspace{5.5cm}
    + \frac{f_t(v')}{\sum_{v:W_{v} \geq \wmax - \eps} f_{t}(v)} \left(\sum_{v:W_v \geq \wmax - \eps} \frac{f^{(-\eps)}_t(v)}{\cZ^{-\eps}_t}\right)
    \\
    & \hspace{3cm} = \frac{f_t(v')}{\sum_{v :W_{v} \geq \wmax - \eps} f_{t}(v)}\left(\sum_{v} \frac{f^{(-\eps)}_t(v)}{\cZ^{-\eps}_t}
- \sum_{v: W_{v} < \wmax - \eps}  \frac{f_{t}(v)}{\cZ_t} \right)\\
& \hspace{3cm} =  \frac{f_t(v')}{\sum_{v :W_{v} \geq \wmax - \eps} f_{t}(v)} \left(1 -  \frac{\sum_{v: W_{v} < \wmax - \eps}f_{t}(v)}{\cZ_t}\right) = \frac{f_t(v')}{\cZ_t},
\end{align*}
\end{linenomath*}
where we use the fact that $\sum_{v} f_{t}(v) = \cZ_t$. Thus, we have $\hat{\cT}_{t+1} \sim \cT_{t+1}$. Moreover, either the same vertex is chosen as the parent of $t+1$ in both $\hat{\cT}^{-\eps}_{t+1}$ and  $\hat{\cT}_{t+1}$, or a vertex of \bluee{weight belonging to $\mathcal{M}_{\eps}$} 
is chosen as the parent of $t+1$ in $\hat{\mathcal{T}}_{t+1}$. This implies the left inequality in \eqref{eq:mono-deg-couple} and in addition, when combined with the fact that \bluee{$g_{-\eps}(W_{t+1}) \leq g(W_{t+1})$}, guarantees that $\cZ^{-\eps}_{t+1} \leq \cZ_{t+1}$. The proof that $\hat{\cT}^{+\eps}_{t+1} \sim \cT^{+\eps}_{t+1}$, the right inequality in \eqref{eq:mono-deg-couple} and $\cZ_{t+1} \leq \cZ^{+\eps}_{t+1}$ are similar, so we may thus iterate the coupling.
\end{proof}
\subsubsection{Proof of Theorem~\ref{thm:cond-edge-dist}} \label{subsub:proof-of-thm-cond} 
In order to prove Theorem~\ref{thm:cond-edge-dist}, we use the auxiliary $\gpaf$-trees $\cT^{+\eps}$ and $\cT^{-\eps}$ according to Lemma~\ref{lem:coupl}.

\begin{proof}[Proof of Theorem~\ref{thm:cond-edge-dist}]
\bluee{For the first assertion, suppose that $B$ is measurable, with $B \subseteq \mathcal{M}^{c}_{\eps}$. Then, if we define the corresponding quantities  $\Xi^{+\eps}(t,\cdot)$, $\Xi^{-\eps}(t,\cdot)$ associated with $\cT^{+\eps}$ and $\cT^{-\eps}$, from the coupling in Lemma~\ref{lem:coupl}, we have
\[
\frac{\Xi^{+\eps}(t,B)}{t} \leq \frac{\Xi(t,B)}{t} \leq \frac{\Xi^{-\eps}(t,B)}{t}.
\] 
Recall that the auxiliary trees $\cT^{+\eps}$ and $\cT^{-\eps}$ have associated functions $g_{+\eps}$ and $g_{-\eps}$ which attain their maxima on a set of positive measure, and thus, satisfy Condition~\hyperlink{c1}{\textbf{C1}}, with Malthusian parameters $\alpha^{(+\eps)} > \lambda^{*}$ and $\alpha^{(-\eps)} > \lambda^{*} - \eps$. Moreover, note that, by the definition of $g_{-\eps}$,
\begin{linenomath*}
\begin{align*}
\E{\frac{h(W)}{\lambda^{*} - g_{-\eps}(W)}} & \leq  \E{\frac{h(W)}{\lambda^* - g(W)}} \leq 1,
\end{align*}
\end{linenomath*}
so that, recalling~\eqref{eq:malth-gpaf}, $\alpha^{(-\eps)} \leq \lambda^*$.  Thus, by Lemma~\ref{lem:coupl}, Theorem~\ref{th:edge-asymptotics} and dominated convergence, almost surely we have 
\[
\limsup_{t \to \infty} \frac{\Xi(t,B)}{t} \leq \lim_{\eps \to 0} \E{\frac{h(W)}{\alpha^{(-\eps)} - g_{-\eps}(W)} \mathbf{1}_{B}(W)}
= \E{\frac{h(W)}{\lambda^{*} - g(W)} \mathbf{1}_{B}(W)}.
\]
Now, we also have $\lim_{\eps \to 0} \alpha^{(+\eps)} = \lambda^{*}$. Indeed, suppose by way of a contradiction that $\lim_{\eps \to 0} \alpha^{(+\eps)} = \alpha' > \lambda^{*})$. Then,
because $\E{h(W)} < \infty$, by dominated convergence we have
\[
1 = \lim_{\eps \to 0} \E{\frac{h(W)}{\alpha^{(+\eps)} - g_{+\eps}(W)}} = \E{\frac{h(W)}{\alpha' - g(W)}}.
\]
But then, \eqref{eq:malth-gpaf} is satisfied for $\lambda$ such that $\lambda^{*} < \lambda < \alpha'$, contradicting the assumption that Condition~\hyperlink{c1}{\textbf{C1}} fails for $\cT$.

It follows that $\lim_{\eps \to 0} \alpha^{(+\eps)} = \lambda^{*}$ and thus, by Lemma~\ref{lem:coupl} and dominated convergence, almost surely we have
\[\liminf_{t \to \infty} \frac{\Xi(t,B)}{t} \leq \lim_{\eps \to 0} \E{\frac{h(W)}{\alpha^{(+\eps)} - g(W)} \mathbf{1}_{B}(W)}
= \E{\frac{h(W)}{\lambda^{*} - g(W)} \mathbf{1}_{B}(W)}.  
\]
The first assertion follows. 
}

For the second assertion, given a measurable set $B$, for each $\eps > 0$, set \bluee{$B^{\eps}:= B \cap \mathcal{M}_{\eps}$}.  
\bluee{Then, by} Lemma~\ref{lem:coupl}, almost surely we have 
\bluee{\begin{linenomath*}
\begin{align*}
\limsup_{t \to \infty} \frac{N_{\geq k} (t, B)}{t} & \leq \liminf_{\eps \to 0} \left(\E{\prod_{i=0}^{k-1} \frac{g_{-\eps}(W)i + h(W)}{g_{-\eps}(W)i + h(W) + \alpha^{(-\eps)}} \mathbf{1}_{B^{\eps}}(W)}
 + \mu(\mathcal{M}_{\eps})\right)
 \\ & = \liminf_{\eps \to 0} \E{\prod_{i=0}^{k-1} \frac{g(W)i + h(W)}{g(W)i + h(W) + \alpha^{(-\eps)}} \mathbf{1}_{B^{\eps}}(W)}
\\ & = \E{\prod_{i=0}^{k-1} \frac{g(W)i + h(W)}{g(W)i + h(W) + \lambda^{*}}\mathbf{1}_{B}(W)}.
 \end{align*}
\end{linenomath*}}
Similarly, almost surely,
\bluee{
\begin{linenomath*}
\begin{align*}
\liminf_{t \to \infty} \frac{N_{\geq k} (t, B)}{t} & \geq \limsup_{\eps \to 0} \E{\prod_{i=0}^{k-1} \frac{g_{+\eps}(W)i + h(W)}{g_{+\eps}(W)i + h(W) + \alpha^{(+\eps)}} \mathbf{1}_{B^{\eps}}(W)} \\ & = \limsup_{\eps \to 0} \E{\prod_{i=0}^{k-1} \frac{g(W)i + h(W))}{g(W)i + h(W) + \alpha^{(+\eps)}} \mathbf{1}_{B^{\eps}}(W)}
\\ & = \E{\prod_{i=0}^{k-1} \frac{g(W)i + h(W))}{g(W)i + h(W) + \lambda^{*}}\mathbf{1}_{B}(W)}.
\end{align*}
\end{linenomath*}}

Finally, for the last assertion, by Lemma~\ref{lem:coupl}, for each $t \in \mathbb{N}_0$ we have 
\[
\frac{\cZ^{-\eps}_t}{t} \leq \frac{\cZ_t}{t} \leq \frac{\cZ^{+\eps}_t}{t}.
\]
Taking limits as $t$ goes to infinity and applying Theorem~\ref{th:strong-law-parti}, the result follows in a similar manner to the previous assertions. 
\end{proof}

\subsection{Degenerate Degrees in the \texorpdfstring{$\gpaf$}{}-tree when Condition~\texorpdfstring{\protect\hyperlink{c1}{\textbf{C1}}}{} Fails} \label{subsec:deg-degs-gpaf}
In this subsection, we show that if the $\gpaf$-tree fails to satisfy Condition~$\hyperlink{c1}{\textbf{C1}}$ by having $m(\lambda, \mathbb{R}_{+}) = \infty$ for all $\lambda > 0$, almost surely the proportion of vertices that are leaves tends to $1$. Consequentially, the limiting mass of edges `escapes to infinity', as described in Theorem~\ref{th:conn-transitions} below. Note that Condition~\textbf{C1} fails in this manner in the $\gpaf$ tree if $\esssup{(g)} = \infty$ or $\E{h(W)} = \infty$. We remark that similar results to Theorem~\ref{th:conn-transitions} have been shown in preferential attachment model with multiplicative fitness with $\mu$ having finite support \cite[Theorem~6]{borgs-chayes} and preferential attachment model with additive fitness (the \emph{extreme disorder} regime in \cite[Theorem~2.6]{bas}. These cases correspond to $h(x) \equiv 0$ and $g(x) \equiv 1$ respectively. 
\redd{In a similar vein to the start of Section~\ref{subsub:gpaf-cond-main}, in this section we will require the following families of sets: for each $m \in \mathbb{N}$, we set 
\begin{equation} \label{eq:g-and-h-upper-trunc}
    \mathscr{G}_{m} := \left\{x: g(x) \geq m\right\}, \: \mathscr{H}_{m} := \left\{x: h(x) \geq m \right\} \: \text{ and } \mathscr{M}_{m} := \mathscr{G}_m \cup \mathscr{H}_m. 
\end{equation}}
\begin{theorem}\label{th:conn-transitions}
Suppose $\cT = (\cT_{t})_{t \geq 0}$ is a $\gpaf$-tree, with associated functions $g,h$, such that $\esssup{(g)} = \infty$ or $\E{h(W)} = \infty$. Then we have the following assertions:
\begin{enumerate}
    \item \redd{For any measurable set $B$, such that for some $m' \in \mathbb{N}$ we have $B \subseteq \mathscr{M}^{c}_{m'}$, 
    \[\frac{\Xi(t,B)}{t} \xrightarrow{t \to \infty} 0, \quad \text{almost surely.}\] }
 \item For any measurable set \redd{$B \subseteq \mathbb{R}_{+}$}, we have
\begin{equation} \label{eq:deg-degs1}
\frac{N_0(t, B)}{t} \xrightarrow{t \to \infty}  \mu(B), \quad \text{almost surely.}
\end{equation} 
 \item The partition function satisfies \[\frac{\cZ_{t}}{t} \xrightarrow{t \to \infty} \infty, \quad \text{almost surely.} \]
\end{enumerate}
\end{theorem}

\begin{proof}
This is similar to the proof of Theorem~\ref{thm:cond-edge-dist}, \redd{however we require some different notation. For each $m \in \mathbb{N}$, let $\cT^{m} = (\cT^{m}_{t})_{t \geq 0}$ and $\cT^{m,m} = (\cT^{m,m}_{t \geq 0})$ denote the tree $\cT$, but modified on the sets $\mathscr{G}_{m}$ and $\mathscr{H}_{m}$. In particular, if we define $g_{m}, h_{m}$ such that 
\[
g_{m} := g\mathbf{1}_{\mathscr{G}_{m}^{c}} + m \mathbf{1}_{\mathscr{G}_{m}} \quad \text{ and } \quad h_{m} := h\mathbf{1}_{\mathscr{H}_{m}^{c}} + m\mathbf{1}_{\mathscr{H}_{m}},
\]
we define $\cT^{m}$ with the associated functions $g_{m}, h$, and $\cT^{m,m}$ with the associated functions $g_{m}, h_{m}$. Then, by mimicking the approach from the coupling in Lemma~\ref{lem:coupl}, for each $m \in \mathbb{N}$ we may couple the processes $(\hat{\cT}^{m,m}, \hat{\cT}^{m}, \hat{\cT})$, such that, for all $t \in \mathbb{N}_0$ their respective partition functions satisfy $\cZ^{m,m}_{t} \leq \cZ^{m}_{t} \leq \cZ_{t}$; for each vertex $v'$ with $W_{v'} \in \mathscr{H}_{m}^{c}$
\[
\outdeg{v', \hat{\cT}^{m}_t} \leq \outdeg{v, \hat{\cT}^{m,m}_{t}}
\]
and for each vertex $v$ with $W_v \in \mathscr{G}_{m}^{c}$
\[
\outdeg{v, \hat{\cT}_t} \leq \outdeg{v, \hat{\cT}^{m}_{t}}.
\]
In this coupling, at each time-step $t$, one samples $\hat{\cT}^{m,m}_t$ first, uses this (with another uniformly distributed random variable) to construct $\hat{\cT}^{m}_t$ and then uses this to construct $\hat{\cT}_{t}$. Therefore, we have the following claim: for a measurable set $B$ let $\Xi^{m,m}(t, B)$ and $N^{m,m}_{\geq k}(t,B)$ denote the counterparts of $\Xi(t,B)$ and $N_{\geq k}(t,B)$ with respect to the tree $\cT^{m,m}$
\begin{claim}
For all $m \in \mathbb{N}$, there exists a coupling $(\hat{\cT}^{m,m}, \hat{\cT})$ of  $\cT^{m,m}$ and $\cT$ such that, on the coupling, for all $t \in \N_{0}$ we have the following:
\begin{enumerate}
    \item For all measurable sets $B \subseteq \mathscr{M}^{c}_{m}$ we have $\Xi(t, B) \leq \Xi^{m,m}(t, B)$;
    \item For all measurable sets $B \subseteq \mathscr{M}^{c}_{m}$ and $k \in \mathbb{N}_{0}$ we have $N_{\geq k}(t, B) \leq N^{m,m}_{\geq k}(t,B)$; 
    \item We have the inequality $\cZ^{m,m}_{t} \leq \cZ_{t}$.
\end{enumerate}    
\end{claim}  
Now note that for all $m$ sufficiently large, $\cT^{m,m}$ satisfies \hyperlink{c1}{\textbf{C1}}: if $\esssup{(g)} = \infty$ then because $\E{h_{m}(W)} \leq m$ and $g_{m}$ attains its maximum $m$ on a set of positive measure, this follows from Remark~\ref{rem:atom-max}. Otherwise, for $m$ sufficiently large we have $g_{m} = g$, and for any $\lambda > \esssup{(g)}$, 
\[
\E{\frac{h_{m}(W)}{\lambda - g(W)}} < \infty, \quad \text{so that, by monotone convergence } \quad \lim_{m \uparrow \infty} \E{\frac{h_{m}(W)}{\lambda - g(W)}} = \infty.
\]
Thus, making $m$ larger if necessary, \hyperlink{c1}{\textbf{C1}} is satisfied for this choice of $\lambda$. In either case, let $\alpha^{(m)}$ denote the Malthusian parameter associated with $\cT^{m,m}$. Then, $\alpha^{(m)} > \esssup{(g_{m})}$ increases as $m$ increases, and, even if $\esssup{(g_{m})} < \infty$ we must have 
\begin{equation} \label{eq:alpha-div}
\lim_{m \uparrow \infty} \alpha^{(m)} = \infty.
\end{equation} 
Indeed, suppose this were not the case, and $\lim_{m \uparrow \infty} \alpha^{(m)} = \alpha' < \infty$. Then, by monotone convergence, 
\[
1 = \lim_{m \to \infty} \E{\frac{h_{m}(W)}{\alpha^{(m)}  - g(W)}} = \E{\frac{h(W)}{\alpha' - g(W)}} = \infty,
\]
since $\E{h(W)} = \infty$. 
Now, the assertions of Theorem~\ref{th:conn-transitions} follow the claim in a similar manner to the way the assertions of Theorem~\ref{thm:cond-edge-dist} follow from Lemma~\ref{lem:coupl}; for brevity we leave these as an exercise to the reader. 
%
%
}
%
\end{proof}
\redd{Now, as in the previous subsection, suppose that $g$ and $h$ are increasing, and $\Supp{\mu} \subseteq [0,\wmax]$, where $\wmax := \sup{\left(\Supp{\mu}\right)}$. The following corollary has a similar proof to Corollary~\ref{cor:cond-weak-conv}, and we therefore also leave it to the reader:
%
%
%
%
\begin{cor} \label{cor-weak-conv-2}
Under the above assumptions, with regards to the weak topology
    \[
\frac{\Xi(t,\cdot)}{t} \xrightarrow{t \to \infty} \delta_{\wmax}(\cdot), \quad \text{almost surely.}
\] 
\end{cor}
}

\section{Analysis of \texorpdfstring{$\rifs{\mu, f, \ell}$}{} assuming \texorpdfstring{\protect\hyperlink{c2}{{\textbf{C2}}}}{}} 
\label{sec:gen-conv}
By Theorem~\ref{th:strong-law-parti}, under certain conditions on the fitness function $f$ and~\hyperlink{c1}{\textbf{C1}}, Condition~\hyperlink{c2}{\textbf{C2}} is satisfied, i.e., 
\[\frac{\mathcal{Z}_t}{t} \xrightarrow{t \to \infty} \alpha, \quad \text{almost surely.}\]
However, Theorem~\ref{thm:cond-edge-dist} shows that this condition~may be satisfied despite Condition~\hyperlink{c1}{\textbf{C1}} failing. Therefore, in this section, we analyse the model under Condition~\hyperlink{c2}{\textbf{C2}}. We state and prove Theorem~\ref{th:general-convergence} below and state Theorem~\ref{th:deg-degs2}, leaving the details to the reader. These proofs rely on Proposition~\ref{prop:conv-of-mean}, proved in Section~\ref{subsec:upper-bound} and Section~\ref{subsec:deducing-convergence}; and Proposition~\ref{prop:conv-second-moment}, proved in Section~\ref{subsec:second-moment}.
\subsection{Main Results: Convergence in probability of \texorpdfstring{$N_{k}(n, B)/\ell n$}{} under \texorpdfstring{\protect\hyperlink{c2}{{\textbf{C2}}}}{}}
\begin{theorem} \label{th:general-convergence}
Assume \textbf{C2}. Then, for any measurable set $B$ we have 
\begin{equation*}
    \frac{N_{k}(t, B)}{\ell t} \xrightarrow{t \to \infty} \E{\frac{\alpha}{f(k,W) + \alpha} \prod_{s=0}^{k-1} \frac{f(s,W)}{f(s,W) + \alpha} \mathbf{1}_{B}(W)} = p^{\alpha}_{k}(B), \quad \text{in probability.}
\end{equation*}
\end{theorem}

In order to prove Theorem~\ref{th:general-convergence}, we define the following \hypertarget{family}{family of sets}:
\begin{equation} \label{eq:determining-class}
\mathscr{F} := \left\{B : B \text{ is measurable and } \forall s \in \N_0, \; f(s,w) \text{ is bounded for $w \in B$} \right\}.
\end{equation}
We also require Proposition~\ref{prop:conv-of-mean} and Proposition~\ref{prop:conv-second-moment}, proved in Section~\ref{subsub-prop} and Section~\ref{subsub-cor}. These proofs rely on the results stated in Section~\ref{subsec:summation} and Section~\ref{subsec:upper-bound}.  
\begin{prop} \label{prop:conv-of-mean}
For any set $B\in \hyperlink{family}{\mathscr{F}}$, for each $k \in \N_0$ we have 
\[\lim_{t \to \infty}
\frac{\E{N_{k}(t,B)}}{\ell t} = p^{\alpha}_{k}(B).\]
\end{prop}

\begin{prop} \label{prop:conv-second-moment}
For any $B \in \hyperlink{family}{\mathscr{F}}$ and $k \in \N_0$ we have \[\lim_{t \to \infty} \E{\frac{\left(N_{k}(t, B)\right)^2}{\ell^2 t^2}} = (p^{\alpha}_{k}(B))^2.\]
\end{prop}

\begin{proof}[Proof of Theorem~\ref{th:general-convergence}]
The result follows for all $B \in \hyperlink{family}{\mathscr{F}}$ by combining Proposition~\ref{prop:conv-of-mean}, Proposition~\ref{prop:conv-second-moment} and applying Chebyshev's inequality. 

Now, let $B$ be an arbitrary measurable set and let $\eps > 0$ be given. Then, since for each $s \in \left\{1, \ldots, k\right\}$ the map $w \mapsto f(s,w)$ is measurable, by Lusin's theorem, we can find a compact set $E \subseteq B$ such that $\mu(B \cap E^{c}) < \eps/3$ and for each $s \in \left\{1, \ldots, k\right\}$ the restriction of the map $w \mapsto f(s,w)$ to $E$ is continuous. Moreover, note that $p^{\alpha}_{k}(B) - p^{\alpha}_{k}(B \cap E) \leq \mu(B \cap E^{c}) < \eps/3$.
Then,
\redd{\begin{linenomath*}
\begin{align} \label{eq:lusin-plus-union-bound}
\nonumber \Prob{\left|\frac{N_{k}(t,B)}{\ell t} - p^{\alpha}_{k}(B)\right| > \eps}
& \leq \mathbb P \bigg(\bigg(\left|\frac{N_{k}(t,B)}{\ell t} - \frac{N_{k}(t,B\cap E)}{\ell t}\right| + \left|\frac{N_{k}(t,B\cap E)}{\ell t} - p^{\alpha}_{k}(B\cap E)\right|  \\ \nonumber & \hspace{6cm}  + \left|p^{\alpha}_{k}(B \cap E) - p^{\alpha}_{k}(B)\right|\bigg) > \eps\bigg) \\ \nonumber & \leq \Prob{\left|\frac{N_{k}(t,B\cap E)}{\ell t} - p^{\alpha}_{k}(B\cap E)\right| > \eps/3}
\\ & \hspace{4.5cm} + \Prob{\left|\frac{N_{k}(t, B)}{\ell t} - \frac{N_{k}(t, B \cap E)}{\ell t}\right| > \eps/3}. 
\end{align}
\end{linenomath*}}
Now, note that by the strong law of large numbers \bluee{applied to $N_{\geq 0}(t,B \cap E^{c})/\ell t$, i.e., the proportion of all vertices with weight belonging to $B \cap E^{c}$,} and Egorov's theorem, for any $\delta > 0$ there exists an event $G$ with $\mathbb{P}(G) < \delta$ such that 
\[\limsup_{t \to \infty} \left(\frac{N_{k}(t, B)}{\ell t} - \frac{N_{k}(t, B \cap E)}{\ell t}\right) = \limsup_{t \to \infty} \frac{N_{k}(t, B \cap E^{c})}{\ell t} \leq \mu(B \cap E^{c})\]
uniformly on the complement of $G$. Therefore, the result follows from~\eqref{eq:lusin-plus-union-bound}, Proposition~\ref{prop:conv-of-mean} and Proposition~\ref{prop:conv-second-moment} by taking limits as $t$ tends to infinity. 
\end{proof}

Using the approach to the upper bound for the mean in the next subsection, and applying Corollary~\ref{cor:summation} stated below with $k=1$ and $e_0, e_1 = 0$, if $N_{\geq 1}(t, B)$ denotes the number of vertices of out-degree at least $1$ in the tree with weight belonging to $B$, we actually have \[\limsup_{t \to \infty} \frac{\E{N_{\geq 1}(t, B)}}{\ell t} \leq \frac{1}{\alpha'}\E{f(0, W) \mathbf{1}_{B}(W)},\]
as long as $\liminf_{t \to \infty} \frac{\cZ_{t}}{t} \geq \alpha'$. By sending $\alpha'$ to infinity, this yields the following analogue of Theorem~\ref{th:conn-transitions}:
\begin{theorem} \label{th:deg-degs2}
Suppose $\cT$ is a \rif{\mu, f, \ell} such that $\lim_{t \to \infty} \frac{\cZ_{t}}{t} = \infty$. Then for any measurable set $B \subseteq [0,\infty)$, we have
\[\frac{N_0(t, B)}{\ell t} \xrightarrow{t\to \infty}  \mu(B), \quad \text{in probability}.\]
\end{theorem}

\subsection{Summation Arguments} \label{subsec:summation}
Here we state some summation arguments required for the subsequent proofs. The following lemma and corollary are taken from \cite{dynamical-simplices}. We include them here, with minor changes in notation, for completeness. 
For $e_0, \ldots, e_k \geq 0, 0 \leq \eta < 1$, let
\[
\mathcal{S}_t(e_0, \ldots, e_k, \eta) := \frac{1}{t} \sum_{\eta t < i_0 < \cdots < i_k \leq t} \prod_{s=0}^{k-1}\left( \left(\frac{i_s}{i_{s+1}} \right)^{e_s} \cdot \frac{1}{i_{s+1} -1} \right) \left(\frac{i_k}{t} \right)^{e_k}.  
\]
\begin{lemma}[{\cite[Lemma~4]{dynamical-simplices}}] \label{lem:prob-sum}
Uniformly in $e_0, \ldots, e_k \geq 0$, $0 \leq \eta \leq 1/2$, we have
\[
\mathcal{S}_t(e_0, \ldots, e_k, \eta) = \prod_{s=0}^{k} \frac{1}{e_s + 1} + \theta(\eta) + O \left(\frac{1}{t^{1/(k+2)}} + \frac{\sum_{s =0}^k e_s \log^{k+1}(t) }{t} \right).
\]
Here, $\theta(\eta)$ is a term satisfying $ |\theta(\eta)| \leq M\eta^{1/(k+2)}$ for some universal constant $M$ depending only on $k$.
\end{lemma}

 \begin{cor}[{\cite[Corollary~5]{dynamical-simplices}}] \label{cor:summation}
For $e_0, \ldots, e_k, f_0, \ldots, f_{k-1} \geq 0$, $0 \leq \eta \leq 1/2$, we have
 \begin{linenomath*}
 \begin{align*}
&\frac 1 t \sum_{\eta t < i_0\leq t} \sum_{\Ind_k \in {\left\{i_0 +1, \ldots, t\right\} \choose k} }  \prod_{s=0}^{k-1}\left( \left(\frac{i_s}{i_{s+1}} \right)^{e_s} \cdot \frac{f_{s}}{i_{s+1} -1} \right) \left(\frac{i_k}{t} \right)^{e_k}\\
& \hspace {2cm} = \frac{1}{e_k +1}\prod_{s=0}^{k-1} \frac{f_{s}}{e_{s}+1}  +  \theta'(\eta) + O \left(\frac{1}{t^{1/(k+2)}}\right).
\end{align*}
\end{linenomath*}
Here, $\theta'(\eta)$ is a term satisfying $ |\theta'(\eta)| \leq M'\eta^{1/(k+2)}$ for some universal constant $M'$ depending only on $k$ and $f_0, \ldots, f_{k-1}$, and the constant in the big $O$-term may depend on $e_0, \ldots, e_{k}, f_0, \ldots, f_k$. 
\end{cor}

\subsection{Upper bound for the Mean of \texorpdfstring{$N_{k}(t,B)/\ell t$}{}} \label{subsec:upper-bound}
In the following subsections, unless otherwise specified, we let $B$ denote an arbitrary element of the family $\hyperlink{family}{\mathscr{F}}$ defined in \eqref{eq:determining-class}.
Let $N_{\eta, k}(t,B)$ be the number of vertices of degree $k \ell$ with weight in $B$ that arrived after time $\eta t$. Then, since $N_{\eta, k}(t,B) \leq N_{k}(t,B) \leq N_{\eta, k}(t,B) + \eta \ell t$, we have
\begin{equation} \label{eq:eta-bound}
\E{\left|\frac{N_{\eta, k}(t, B)}{\ell t} - \frac{N_{k}(t, B)}{\ell t}\right|} \leq \eta.
\end{equation}
Thus, to obtain an upper bound for the convergence of the mean, it suffices to prove that \[\limsup_{\eta \to 0} \limsup_{t \to \infty} \E{\frac{N_{\eta, k}(t, B)}{\ell t}} = p^{\alpha}_{k}(B).\] 
In what follows, we use the notation $d_{i}(t)$ to denote the out-degree  at time $t$ of the vertex $i$ born at time $i_0 := \lfloor i/\ell \rfloor$.
We then have,
\[\E{N_{\eta, k}(t,B)} = \sum_{\eta t < i_{0} \leq t - k} \ell \cdot \Prob{d_{i}(t) = k, W_i \in B}, \]
since the probability is identical for each of the $\ell$ vertices born at each time $i_0$. 
In what follows, for a given $i$ we denote by $\mathcal{I}_{k} : = \{i_1, \ldots, i_k\}$ a collection of natural numbers $i_0 < i_1 < \ldots < i_k \leq t$. For ease of notation we exclude the dependence of $\mathcal{I}_{k}$ on $i$. 

For a natural number $s > i_0$, we use the notation $i \bluee{\rightarrow} s$ to denote that $i$ is the vertex chosen at the $s$th time-step, hence $i$ gains $\ell$ new neighbours at time $s$. Likewise, the notation $i \not \bluee{\rightarrow} s$ denotes that $i$ is not chosen at the $s$th time-step. 
Then, let $\cE_{i}(\Ind_k, B)$ denote the event that $W_i \in B$ and for all $s \in \left\{i_0 +1, \ldots, t\right\}$, $i \bluee{\rightarrow} s$ if and only if $s \in \Ind_k$. Clearly, we have 
\[ 
\Prob{d_{i}(t) = k, W_i \in B} = \sum_{\Ind_k \in {\left\{i_0 +1, \ldots, t\right\} \choose k}} \Prob{\cE_{i}(\Ind_k, B)}.
\]
where ${\left\{i_0 +1, \ldots, t\right\} \choose k}$ denotes the set of all subsets of $\left\{i_0 +1, \ldots, t\right\}$ of size $k$.
For $\eps > 0$ and $t \geq 0$ and natural numbers $N_1 \leq N_2$, we let 
\begin{equation} \label{def:a}
\mathcal G_\eps (t)= \left\{ \left| \mathcal{Z}_t- \alpha t \right| < \eps \alpha t \right\}, \text{ and } \mathcal G_\eps(N_1, N_2) = \bigcap_{t=N_1}^{N_2}\mathcal G_\eps(t). 
\end{equation}
Moreover, for $t \geq 1$, we denote by  $\mathscr T_{t}$ the $\sigma$-field generated by $(\cT_s)_{1 \leq s \leq t}$, containing all the information generated by the process up to time $t$. By the assumption of almost sure convergence and Egorov's theorem, for any $\delta, \eps > 0$, there exists $N' = N'(\eps, \delta)$ such that, for all $t \geq N'$,  $\Prob{\mathcal G_\varepsilon(N', t)} \geq 1- \delta$. Thus, for $t \geq N'/\eta$, we have
\begin{linenomath*}
\begin{align} \label{eq:exp-approx}
& \E{N_{\eta, k}(t,B)} \leq \E{N_{\eta, k}(t,B) \mathbf{1}_{\mathcal{G}_{\eps}(N',t)}} + \ell t\left(1-\Prob{\mathcal G_\varepsilon(N', t)} \right) \\ 
& \hspace{3cm} \nonumber \leq \ell \left(\sum_{\eta t < i_0 \leq t} \sum_{\Ind_k \in {\left\{i_0 +1, \ldots, t\right\} \choose k}} \Prob{\cE_{i}(\Ind_k, B) \cap \mathcal{G}_{\eps}(i_0,t)} + \delta t \right). 
\end{align}
\end{linenomath*}
We use the shorthand $\alpha_{\pm \eps} := (1 \pm \eps) \alpha$. 

\begin{prop} \label{prop:upper_bounds}
Let $B \in \hyperlink{family}{\mathscr{F}}$ and $0 < \varepsilon, \eta \leq 1/2$. As $t \to \infty$, uniformly in $\eta t < i_0 \leq t -k, \Ind_k = \{i_1, \ldots, i_k\}\in{\left\{i_0 +1, \ldots, t\right\} \choose k}$ and the choice of $\varepsilon$, we have
\begin{linenomath*}
\begin{align*}
& \Prob{\cE_{i}(\Ind_k, B) \cap \mathcal{G}_{\eps}(i_0,t)}
\\ & \hspace{2cm} \leq \left(1 + O(1/t) \right) \E{\left(\frac{i_{k}}{t}\right)^{f(k,W)/\alpha_{+\eps}} \prod_{s=0}^{k-1} \left( \frac{i_{s}}{i_{s +1}}\right)^{f(s, W)/\alpha_{+\eps}} \frac{f(s,W)}{\alpha_{-\eps}(i_{s+1} - 1)} \mathbf{1}_{B}(W)}. 
\end{align*}
\end{linenomath*}
\end{prop}

\begin{cor} \label{cor:1}
Let $B \in \hyperlink{family}{\mathscr{F}}$ and $0 < \delta, \varepsilon, \eta \leq 1/2$. Then, there exists $N = N(\delta, \varepsilon, \eta)$ such that, for all $t \geq N$,
\[
\frac{\E{N_{\eta, k}(t,B)}}{\ell t} \leq (1+\delta) \left(\frac{1+\eps}{1-\eps}\right)^k \E{ \frac{\alpha_{+\eps}}{f(k, W)+\alpha_{+\eps}} \prod_{s=0}^{k-1} \frac{f(s, W)}{f(s, W)+\alpha_{+\eps}} \mathbf{1}_{B}(W)} + C \eta^{1/(k+2)} +  \delta,
\]
where the constant $C$ may depend on $k$ and $B$ but not on $n$ and not on the choices of $\delta, \varepsilon, \eta$.
In particular, for each $B \in \mathscr{F}$ and $k \in \mathbb{N}_0$,
\[\limsup_{t \to \infty} \E{N_k(t,B)}/ \ell t \leq p^{\alpha}_k(B).\]
\end{cor}
\begin{proof}
This follows from applying~\eqref{eq:exp-approx} and Proposition~\ref{prop:upper_bounds} and then applying Corollary~\ref{cor:summation} with $e_j = f(j,W)/\alpha_{+\eps}$ and $f_j = f(j,W)/\alpha_{-\eps}$ to bound the sum over the collection of indices. Note that the term $\left(\frac{1 + \eps}{1-\eps}\right)^k$ comes from replacing $\alpha_{-\eps}$ by $\alpha_{+\eps}$. 
\end{proof}
We proceed towards the proof of Proposition~\ref{prop:upper_bounds}. Let $\eps, \eta$ be given such that $0 < \varepsilon, \eta \leq 1/2$. For $\eta t < i_0 \leq t $ and
$\Ind_k = \{i_1, \ldots, i_k\} \in {\left\{i_0 +1, \ldots, t\right\} \choose k}$ for each $s \in \left\{i_0 +1, \ldots, t\right\}$, we define 
\[ 
\cD_{s} := 
\begin{cases}
\left\{i\bluee{\rightarrow} s\right\}, & \text{ if } s\in\Ind_k,\\
\left\{i\bluee{\not \rightarrow} s\right\}, & \text{ otherwise},
\end{cases} \]
and $\tilde \cD_s  = \cD_s \cap \mathcal G_{\eps}(s).$ We also define $\tilde \cD_{i_0} = \mathcal G_\eps(i_0) \cap \{W_i \in B\}$, and 
for simplicity of notation, write $D_j$ and $\tilde D_j$ for the indicator random variables $\mathbf{1}_{\cD_j}$ and $\mathbf{1}_{\tilde \cD_j}$ respectively.
Note that $\mathcal E_i(\Ind_k, B) \cap \mathcal G_\eps(i_0,t) = \bigcap_{j=i_0}^t \tilde \cD_j$. To bound the probability of this event, we define
\[ 
X_s = \E{\prod_{j=i_{s}+1}^{t} \tilde D_j  \; \bigg | \; \mathscr T_{i_s}}  \tilde D_{i_s}, 
\quad s \in \left\{0, \ldots, k\right\}\]
and observe that $\E{X_0} =  \Prob{\bigcap_{s=i_0}^t \tilde \cD_s}$ is the sought after probability. 

\begin{lemma} \label{lem:u}
For $s \in \left\{0, \ldots, k\right\}$, we have
\begin{equation} \label{eq:backwards}
X_s \leq  \prod_{u=i_{k}+1}^{n} \left(1 - \frac{f(k,W)}{\alpha_{+\eps}(u- 1)}\right) \left(\prod_{j = s}^{k-1} \frac{f(j,W)}{\alpha_{-\eps}(i_{j+1}-1)}\prod_{j'= i_j +1}^{i_{j+1} -1} \left(1 - \frac{f(j,W)}{\alpha_{+\eps}(j' -1)}\right)\right)\tilde{D}_{i_s},
\end{equation}
where we interpret any empty products (for example when $i_k = n$) as equal to $1$. In particular, 
\begin{align}
\label{eq:upp-bound}
\E{X_0} &\leq \E{\prod_{u=i_{k}+1}^{n} \left(1 - \frac{f(k,W)}{\alpha_{+\eps}(u- 1)}\right) \left(\prod_{j = 0}^{k-1} \frac{f(j,W)}{\alpha_{-\eps}(i_{j+1}-1)} \prod_{j'= i_j +1}^{i_{j+1} -1} \left(1 - \frac{f(j,W)}{\alpha_{+\eps}(j' -1)}\right)\right) \mathbf{1}_{B}(W)}. \qquad
\end{align}
\end{lemma}

\begin{proof}
We prove \eqref{eq:backwards} by backwards induction. For the base case, $s=k$, if $i_k = n$, the inequality is trivial, as $X_k = \tilde{D}_{i_k}$. Thus, assuming $i_k < n$, by the tower property,
\begin{linenomath*}
\begin{align*}
  \E{\prod_{j=i_{k}+1}^{n} \tilde D_j \; \bigg  | \; \mathscr T_{i_k}} & = \E{\E{\tilde D_{n} \; \bigg | \; \mathscr{T}_{n-1}}\prod_{j=i_k+1}^{n-1} \tilde D_j \; \bigg | \; \mathscr T_{i_k}}  \\ &\leq \E{\E{D_{n} \; \bigg | \; \mathscr{T}_{n-1}}\prod_{j=i_k+1}^{n-1} \tilde D_j \; \bigg | \; \mathscr T_{i_k}} \\ &= \E{\left(1 - \frac{f(k,W)}{\mathcal{Z}_{n-1}}\right)\prod_{j=i_k+1}^{n-1} \tilde D_j \; \bigg | \; \mathscr T_{i_k}} \\ &\leq \left(1 - \frac{f(k,W)}{ \alpha_{+\eps}(n-1)}\right) \E{\prod_{j=i_k+1}^{n-1} \tilde D_j \; \bigg | \; \mathscr T_{i_k}},
\end{align*}
\end{linenomath*}
and iterating this argument with the conditional expectation on the right hand side proves the base case. Now, note that for $s \in \left\{0, \ldots, k-1\right\}$ 
\[
    X_s = \E{X_{s+1} \prod_{j=i_{s} +1}^{i_{s+1} -1} \tilde D_{j} \; \bigg | \; \mathscr T_{i_s}} \tilde D_{i_s}.
\]
Applying the induction hypothesis, it suffices to bound the term 
$\E{\prod_{j=i_{s} +1}^{i_{s+1}} \tilde D_{j} \; \bigg | \; \mathscr T_{i_s}}$, and, similar to the base case, we may assume $i_{s} < i_{s+1} - 1$. 
But, then, we have 
\begin{linenomath*}
\begin{align*}
\E{\prod_{j=i_{s} +1}^{i_{s+1}} \tilde D_{j} \; \bigg | \; \mathscr T_{i_s}} & = \E{\E{\tilde D_{i_{s+1}} \; \bigg | \; \mathscr{T}_{i_{s+1} -1}} \prod_{j=i_{s} +1}^{i_{s+1}-2} \tilde D_{j} \; \bigg | \; \mathscr T_{i_s}} \\ &\leq \E{\E{D_{i_{s+1}} \; \bigg | \; \mathscr{T}_{i_{s+1} -1}} \prod_{j=i_{s} +1}^{i_{s+1}-2} \tilde D_{j} \; \bigg | \; \mathscr T_{i_s}}
\\ &\leq \frac{f(s,W)}{\alpha_{-\eps}(i_{s+1} - 1)}\E{\prod_{j=i_s + 1}^{i_{s+1}-2} \tilde{D}_{j} \; \bigg | \; \mathscr T_{i_s}}
\\ & \leq \frac{f(s,W)}{\alpha_{-\eps}(i_{s+1} - 1)}\E{\E{D_{i_{s+1} -1} \; \bigg | \; \mathscr{T}_{i_{s+1} -1}} \prod_{j=i_{s} +1}^{i_{s+1}-2} \tilde D_{j} \; \bigg | \; \mathscr T_{i_s}} 
\\ & \leq \frac{f(s,W)}{\alpha_{-\eps}(i_{s+1} - 1)}\left(1 - \frac{f(s,W)}{\alpha_{+\eps} \left(i_{s+1} -2\right) } \right) \E{ \prod_{j=i_{s} +1}^{i_{s+1}-2} \tilde D_{j} \; \bigg | \; \mathscr T_{i_s}}.
\end{align*}
\end{linenomath*}
Iterating these bounds the inductive step follows in a similar manner to the base case. Finally, noting that $\mathbf{1}_{\tilde{\mathcal{D}}_{i}} \leq \mathbf{1}_{B}(W)$ proves \eqref{eq:upp-bound}. 
\end{proof}

The next lemma follows from a simple application of Stirling's formula, i.e.,~\eqref{eq:stirling_gamma_approx}:
\begin{lemma} \label{lem:simple_stir}
Let $\eta, C > 0$. Then, uniformly over $\eta t \leq a \leq b$ and $0 \leq \beta \leq C$, we have
\[\prod_{j=a+1}^{b-1} \left(1-\frac{\beta}{j-1}\right) =  \left(\frac{a}{b}\right)^\beta \left( 1 + O\left(\frac 1 t\right) \right).\]
\end{lemma}
\qed
\begin{proof}[Proof of Proposition~\ref{prop:upper_bounds}]
We take the upper bound $\E{X_0}$ from Lemma~\ref{lem:u} and bound each of the products by applying Lemma~\ref{lem:simple_stir}.
\end{proof}
\subsection{Deducing Convergence of the Mean of \texorpdfstring{$N_{k}(t,B)/\ell t$}{}} \label{subsec:deducing-convergence}
In this subsection we deduce a lower bound on $\liminf_{t \to \infty} \E{N_k(t,B)}/ \ell t$
on measurable sets $B \in \hyperlink{family}{\mathscr{F}}$.
In what follows, denote by $N_{\geq M}(t,B)$ the number of vertices of out-degree $\geq \ell M$ with weight belonging to $B$. Moreover, let $N(t,B) = N_{\geq 0}(t,B)$ denote the total number of vertices at time $t$ with weight belonging to $B$. 
\begin{lemma} \label{lem:big-mass-dies}
For any measurable set $B$, we have, $\limsup_{t \to \infty} \frac{N_{\geq M}(t,B)}{\ell t} \leq \frac{1}{M}$ almost surely.
\end{lemma}
\begin{proof}
Since we add $\ell$ vertices at each time-step, we have $\limsup_{t \to \infty} \frac{\left|\mathcal{T}_{t}\right|}{\ell t} = 1$. However, $\left | \mathcal{T}_t \right | \geq M N_{\geq M}(t, \mathbb{R})$, since the right-side only provides a lower bound for the number of vertices in the tree incident to those with out-degree at least $M$. The result follows by dividing both sides by $M \ell t$ and sending $t$ to infinity. 
\end{proof}

\subsubsection{Proof of Proposition~\ref{prop:conv-of-mean}}
\label{subsub-prop}
\begin{proof}
Recall that Corollary~\ref{cor:1} showed that for each $B \in \mathscr{F}$ and $k \in \mathbb{N}_0$,
\[\limsup_{t \to \infty} \E{N_k(t,B)}/ \ell t \leq p^{\alpha}_k(B).\]
Now, suppose that Proposition~\ref{prop:conv-of-mean} fails, so that, in particular there exists some set $B' \in \mathscr{F}$ and $k' \in \mathbb{N}_{0}$ such that 
\[ \liminf_{t \to \infty} \frac{\E{N_{k'}(t,B')}}{\ell t} < p^{\alpha}_{k'}(B').\]
Thus, for some $\epsilon' > 0$, we have 
$\liminf_{t \to \infty} \frac{\E{N_{k'}(t,B')}}{\ell t} \leq p^{\alpha}_{k'}(B) - \epsilon'$. Now, using Lemma~\ref{lem:big-mass-dies}, choose $M > \max{\left \{k', \frac{2}{\epsilon'}\right \}}$, so that $\limsup_{t \to \infty} \frac{N_{\geq M}(t,B')}{\ell t} < \epsilon'/2$. Then, recalling~\eqref{eq:to-delete},
\begin{linenomath*}
\begin{align}
\label{eq:lower1}
\liminf_{t \to \infty} \E{\sum_{k=0}^{M} \frac{N_{k}(t,B')}{\ell t}} & \leq \liminf_{t \to \infty} \E{\frac{N_{k'}(t,B')}{\ell t}} + \sum_{k \neq k'} \limsup_{t \to \infty} \E{\frac{N_{k}(t,B')}{\ell t}} 
\\ \nonumber & \leq \left(\sum_{k=0}^{\infty} p^{\alpha}_{k}(B')\right) - \epsilon' \leq \mu(B') - \epsilon'.  
\end{align}
\end{linenomath*}
On the other hand, by Fatou's Lemma, we have 
\begin{linenomath*}
\begin{align}
\label{eq:lower2}
\liminf_{t \to \infty} \E{\sum_{k=0}^{M} \frac{N_{k}(t,B')}{\ell t}} & \geq \E{\liminf_{t \to \infty} \sum_{k=0}^{M} \frac{N_{k}(t,B')}{\ell t}} \\ \nonumber &= \E{\liminf_{t \to \infty}\left(\frac{N(t,B')}{\ell t} - \frac{N_{\geq M}(t,B')}{\ell t}\right)}
\geq \mu(B') - \epsilon'/2,
\end{align}
\end{linenomath*}
where the last inequality follows from the strong law of large numbers.
But then, combining~\eqref{eq:lower1} and \eqref{eq:lower2}, we have $\mu(B') - \epsilon' \geq \mu(B') - \epsilon'/2$, a contradiction.
\end{proof}

\subsection{Second Moment Calculations} \label{subsec:second-moment}
In order to bound the second moment, we apply similar calculations to the start of the section to compute asymptotically the number of pairs of vertices of out-degree $k \ell$ born after time $\eta t$. 
For vertices $i$ and $j$, as in Section~\ref{subsec:upper-bound}, we set $i_0 := \lfloor i/\ell \rfloor$ and $j_0 := \lfloor j/\ell \rfloor$, and note that
\begin{equation} \label{eq:second-moment-sum}
\E{\left(N_{\eta, k}(t,B)\right)^2} = \sum_{\eta t < i_0, j_0 \leq t-k} \sum_{j: \lfloor j/\ell \rfloor = j_0} \sum_{i: \lfloor i/\ell \rfloor = i_0 } \Prob{d_{i}(t) = k, W_i \in B, d_{j}(t) = k,  W_j \in B}. 
\end{equation}
 Note that, in a similar manner to~\eqref{eq:eta-bound}, we have
\[\E{\left|\frac{\left(N_{\eta, k}(t,B)\right)^2}{\ell^2 t^2} - \frac{\left(N_{k}(t,B)\right)^2}{\ell^2 t^2}\right|} \leq \eta \]
so that it suffices to prove that \[\limsup_{\eta \to 0} \limsup_{t \to \infty} \E{\frac{\left(N_{\eta, k}(t, B)\right)^2}{\ell^2 t^2}} \leq (p^{\alpha}_{k}(B))^2.\]

Recall that, for a given $i$ we denote by $\Ind_k$ a collection of natural numbers $i_0 < i_1 < \cdots < i_k \leq t$. Moreover, for a given $j$, we denote by $\Jind_{k}$ a collection of natural numbers $j_0 < j_1 < \cdots < j_k \leq t$. Similar to Section~\ref{subsec:upper-bound}, for $s > j$ we use the notation $j \bluee{\rightarrow} s$ to denote that $j$ is the vertex chosen at the $s$th time-step and likewise, we let $\cE_{j}(\Jind_k, B)$ denote the event that $W_j \in B$ and for all $s \in \left\{j_0 +1, \ldots, t\right\}$, $j \bluee{\rightarrow} s$ if and only if $s \in \Jind_k$. Then we have
\[
\Prob{d_{i}(t) = k, W_i \in B, d_{j}(t) = k,  W_j \in B} = \sum_{\Jind_k \in {\left\{j_0 +1, \ldots, t\right\} \choose k}} \sum_{\Ind_k \in {\left\{i_0 +1, \ldots, t\right\} \choose k}} \Prob{\cE_i(\Ind_k,B) \cap \cE_j(\Jind_k, B)}.
\]
Note that the contribution to the above sum corresponding to terms with $\Ind_{k} \cap \Jind_{k} \neq \varnothing$, and $i \neq j$, is zero, since it is impossible for distinct vertices to be chosen in a single time-step. But then, the terms corresponding to $i = j$ contribute at most $\E{N_{\eta,k}(n, B)} \leq \ell n$ to the right side of~\eqref{eq:second-moment-sum}. Next, for any choice of indices with $\bluee{\Ind_{k} \cap \Jind_{k} = \emptyset}$, there are at most $\ell^2$ pairs of vertices $(i,j)$ born at respective times $(i_0, j_0)$ contributing to the sum in~\eqref{eq:second-moment-sum}.
Recalling the definitions of $\mathcal{G}_{\eps}(t), \mathcal{G}_{\eps}(N_1,N_2)$ and $N' = N'(\eps, \delta)$ from \eqref{def:a} and below in the previous subsection, in a similar manner to~\eqref{eq:exp-approx} we have, for $t \geq N'/\eta$,
\begin{equation} \label{eq:sec-moment}
\E{\left(N_{\eta, k}(t, B) \right)^2} \leq \ell^2 \left( \sum_{\eta t < i_0, j_0 \leq t-k} \sum_{\bluee{\Ind_{k} \cap \Jind_{k} = \emptyset}}  \Prob{\cE_i(\Ind_k,B) \cap \cE_j(\Jind_k, B) \cap \mathcal{G}_{\eps}(i_0, t)} + \delta t^2\right) + \ell t.
\end{equation}
We then have the following:
\begin{prop} \label{prop:upper-bounds-sec}
Let $B \in \hyperlink{family}{\mathscr{F}}$ and $0 < \eps, \eta \leq 1/2$. As $t \to \infty$, uniformly in $\eta t < i_0 \leq j_0 \leq t -k$ and $\Ind_k \in {\left\{i_0 +1, \ldots, t\right\} \choose k}$, $\Jind_k \in {\left\{j_0+1, \ldots, t\right\} \choose k}$ such that 
$\bluee{\Ind_{k} \cap \Jind_{k} = \emptyset}$, and the choice of $\eps$, we have
\begin{linenomath*}
\begin{align} \label{eq:prop-second-moment}
& \nonumber\Prob{\cE_i(\Ind_k,B) \cap \cE_j(\Jind_k, B) \cap \mathcal{G}_{\eps}(i_0, t)} \\ \nonumber 
&\hspace{2cm} \leq (1 + O(1/t)) \E{\left(\frac{i_{k}}{t}\right)^{f(k,W)/\alpha_{+\eps}} \cdot \prod_{s=0}^{k-1} \left(\left( \frac{i_{s}}{i_{s +1}}\right)^{f(s, W)/\alpha_{+\eps}} \frac{f(s,W)}{\alpha_{-\eps}(i_{s+1}-1)}\right) \mathbf{1}_{B}(W)}
\\ & \hspace{4.4cm} \times \E{\left(\frac{j_{k}}{t}\right)^{f(k,W)/\alpha_{+\eps}} \cdot \prod_{s=0}^{k-1} \left(\left( \frac{j_{s}}{j_{s +1}}\right)^{f(s, W)/\alpha_{+\eps}} \frac{f(s,W)}{\alpha_{-\eps}(j_{s+1}-1}\right)\mathbf{1}_{B}(W)}.
\end{align}
\end{linenomath*}
\end{prop}
We leave the details of the proof of this proposition to the reader, as it follows an analogous approach to the proof of Proposition~\ref{prop:upper_bounds}, using a backwards induction argument. 
\begin{proof}[Proof Sketch]
Let $u_1, \ldots, u_{2k}$ denote the indices in $\Ind_k \cup \Jind_k$, and $f_x(i), f_{x}(j)$ denote the fitnesses associated with vertex $i$ and vertex $j$ at time $x$.  Then, when we bound the probabilities $\{i  \bluee{\not \rightarrow} x \} \cap \{j  \bluee{\not\rightarrow} x\}$ for all $x \in \left\{u_{s} +1, \ldots, u_{s+1} - 1\right\}$ from above we obtain terms of the form
\[
\prod_{x = u_{s}+1}^{u_{s+1} - 1} \left(1 - \frac{f_{x}(i) + f_{x}(j)}{\alpha_{+\eps}(x-1)} \right) = \left(\frac{u_{s}}{u_{s+1}} \right)^{f_{x}(i) + f_{x}(j)} \left(1+ O\left(\frac{1}{t}\right)\right),
\]
where the right side follows from Lemma~\ref{lem:simple_stir}. Then, when we evaluate the expectation analogous to the expectation appearing in~\eqref{eq:upp-bound}, we obtain an expectation involving products of terms dependent on $W_i$ and $W_j$, i.e., the weights associated with vertex $i$ and vertex $j$. These terms separate into a product of expectations by the independence of the random variables $W_i$, $W_j$, and finally, many of the products telescope to yield the right side of~\eqref{eq:prop-second-moment}.
\end{proof}

\subsubsection{Proof of Proposition~\ref{prop:conv-second-moment}}
\label{subsub-cor}
\begin{proof}
We apply Proposition~\ref{prop:upper-bounds-sec} to bound the summands in \eqref{eq:sec-moment}. Then, as we are looking for an upper bound, we may drop the condition $\bluee{\Ind_{k} \cap \Jind_{k} = \emptyset}$ when evaluating the sum. But then, by Corollary~\ref{cor:summation}, we have, uniformly in $\eps$ and $\eta$,
\begin{linenomath*}
\begin{align*}
& \sum_{\eta t < i_0, j_0 \leq t} \sum_{\Ind_k, \Jind_k} \E{\left(\frac{i_k}{t}\right)^{f(k,W)/\alpha_{+\eps}}\cdot \prod_{s=0}^{k-1} \left(\frac{i_s}{i_{s+1}}\right)^{f(s,W)/\alpha_{+\eps}}\frac{f(s,W)}{\alpha_{-\eps} (i_{s+1}-1)} \mathbf{1}_{B}(W)} \nonumber \nonumber \\ &\hspace{2.5cm}
\times \E{\left(\frac{j_k}{t}\right)^{f(k,W)/\alpha_{+\eps}} \cdot \prod_{s=0}^{k-1} \left(\frac{j_s}{j_{s+1}}\right)^{f(s,W)/\alpha_{+\eps}}\frac{f(s,W)}{\alpha_{-\eps} (j_{s+1}-1)} \mathbf{1}_{B}(W)} \\
& \leq \left(\frac{1+\eps}{1-\eps}\right)^{2k} \left( \E{\frac{\alpha_{+\eps}}{f(k,W)+\alpha_{+\eps}} \prod_{s=0}^{k-1} \frac{f(s,W)}{f(s,W)+\alpha_{+\eps}} \mathbf{1}_{B}(W)} \right)^2+ O\left(t^{-1/(k+2)}\right) + C'\eta^{1/{k+2}},
\end{align*}
\end{linenomath*}
for some universal constant $C' >0$, depending only on $B, f$. The result follows. 
\end{proof}
\redd{
 \section*{Acknowledgements}
 I would like to thank my supervisor Nikolaos Fountoulakis for his guidance and providing useful feedback on earlier drafts. I would also like to thank C\'{e}cile Mailler and Henning Sulzbach, whose collaborative work (along with Nikolaos) on a previous project introduced me to some of the techniques used in this paper. Finally, I would like to thank the anonymous referees for their helpful comments, which greatly improved the presentation of this paper, including the removal of some unnecessary assumptions in the statements of Theorem~\ref{thm:cond-edge-dist} and Theorem~\ref{th:conn-transitions}.}
\bibliographystyle{amsplain}
\bibliography{ref}

\end{document}